\documentclass[12pt]{article}
\usepackage[english]{babel}
\usepackage{epsfig}
\usepackage{amsfonts}
\usepackage{amssymb}
\usepackage{amsmath}
\usepackage{amsthm}
\usepackage{latexsym}
\usepackage{graphicx}
\usepackage{bm}
\usepackage{tabularx}
\usepackage{booktabs} 
\usepackage{enumerate}
\usepackage{caption}
\usepackage{wrapfig}  
\usepackage{geometry}
\usepackage{mathtools}  
\usepackage{eqnarray}
\usepackage{enumitem}

\usepackage{graphicx}
\usepackage{subcaption}

\usepackage[title]{appendix}

\usepackage[numbers]{natbib}
\usepackage{comment}

\usepackage{bbm} 
\usepackage{url}
\usepackage{hyperref}
\usepackage{color}
\usepackage{float}
\DeclareMathOperator{\sign}{sgn}

\usepackage{hyperref}
\hypersetup{colorlinks=true, citecolor=blue, linkcolor=red, urlcolor=green}
 
\usepackage{bbm}
\usepackage{url}
\usepackage{color}
\usepackage{xcolor}
\usepackage[section]{placeins}

\graphicspath{{New_Graphs/}}

\theoremstyle{plain}
\newtheorem{theorem}{Theorem}[section]
\newtheorem{lemma}[theorem]{Lemma}

\newtheorem{proposition}[theorem]{Proposition}

\theoremstyle{definition}

\theoremstyle{remark}

\newcommand{\shortdot}[1]{\raisebox{-0.4pt}{$\stackrel{\bullet}{#1}$}}

\title{The Amplitude Dynamics of Impulsive Queues}
\author{ Ruici Gao \\
School of Operations Research and Information Engineering
\\
Cornell University
\\
{rg585@cornell.edu} \\ \\
Jamol Pender \footnote{Corresponding Author}
\\
School of Operations Research and Information Engineering
\\
Cornell University
\\
{jjp274@cornell.edu} 
}

\usepackage{natbib}
\usepackage{graphicx}
\usepackage{amsmath}
\date{}

\begin{document}

\maketitle

\begin{abstract}
In this paper, we analyze a multiserver Markovian queue with customer abandonment i.e. the Erlang-A queue uder a novel framework, i.e. the random impulsive differential equations (RIDEs). This framework captures systems that evolve continuously while experiencing sudden, discrete interventions. The combination of such framework with Erlang-A queue give rise to multiple real-life applications, such as improving efficiency at call centers and designing optimal timing to apply quantum error correction in trying to shun away from decoherence in quantum computing. We derive closed-form expressions for steady-state amplitude bounds and average queue lengths under impulse, and we identify the impulse timings that optimize system performance.
\end{abstract}

\section{Introduction}

The $M_{\lambda}/M_{\mu}/c+M_{\theta}$ queue or the Erlang-A queue is now a classic model in queueing theory. This queueing model has been important for modeling complex service systems where customers can abandon the system if they are impatient. Despite the fact that many papers have been written on this queueing model there is no paper that computes the steady state distribution for this queueing model in closed form, the distribution of the waiting time in closed form, or the conditional probability of abandonment. 

One interesting application of the Erlang-A queue is in call centers.  In this setting, customers call the call center to speak with agents about any problems they are having.  However, in some call centers, while the customers are waiting, they here music or are reminded of their wait by an automated message.  

Customer abandonment frequently occurs in call center queues, often exacerbated by automated messages that highlight prolonged wait times. While intended to reassure callers, these messages can inadvertently magnify their frustration. Explicitly stating the current wait duration or implying a slow-moving queue can amplify feelings of impatience and dissatisfaction. This is particularly detrimental when the message lacks crucial information such as an estimated resolution time or alternative service options like a callback feature. Faced with mounting frustration and uncertainty, customers may choose to disconnect, perceiving the wait as an inefficient use of their time, especially when competing priorities or alternative solutions are available. This phenomenon underscores the critical need for carefully crafted communication that maintains caller engagement while mitigating frustration.

The need for call center models to incorporate abandonment behavior is crucial. By accurately predicting and understanding the factors that drive customer abandonment, call centers can:
\begin{itemize}
    \item Optimize resource allocation: Adjust staffing levels and call routing strategies in real-time to minimize wait times and reduce the likelihood of abandonment.
    \item Proactively address customer concerns: Implement proactive measures such as offering callbacks or alternative service channels to prevent frustration and maintain customer satisfaction.
    \item Improve service quality: Analyze abandonment data to identify root causes of long wait times and implement process improvements to enhance overall service efficiency.
    \item Enhance customer experience: Tailor communication strategies to better address customer concerns and provide a more positive and efficient customer journey.
\end{itemize}
By acknowledging and addressing customer abandonment behavior, call centers can significantly improve their operational efficiency, enhance customer satisfaction, and ultimately strengthen customer relationships.

Another intriguing application lies in the field of quantum error correction in quantum computing. Quantum error correction (QEC) is a mechanism that protects fragile quantum information against decoherence by performing periodic checks to detect and correct errors before they accumulate catastrophically. This theory is proposed and laid foundation by the works of \citet{shor1995} and \cite{steane1996} and refined by multiple works afterwards (\citet{terhal2015,Siegel2023adaptivesurfacecode, kelly2015state}). 

A key practical challenge in QEC is balancing the trade-off between correction frequency and resource overhead. Frequent corrections reduce the chance of multiple errors accumulating into uncorrectable logical failures, thereby lowering the steady-state logical error rate. However, each correction cycle consumes time and computational resources, introduces additional measurement and gate operations, and may disrupt other parts of the computation as discussed in \cite{Gidney2019efficientmagicstate}. Another aspect of QEC to be noticed is that, most architectures of QEC, including the leading practical architecture surface codes, proposed by \cite{fowler2012}, inherently imposes the QEC mechanism to the system periodically. This prompts us to investigate how to optimally time these impulsive interventions. This paper novelly models the build-up and removal of errors as a random impulsive differential equation with underlying erlang-a queues as the nature of Erlang-A queues suits the scenario perfectly. Error buildups can be modeled as arrivals, gate operations as servers, and QEC resetting logical error state serving as the impulse applied. Moreover, in quantum systems, when a qubit spends too much time in the system and accumulated errors faster then the code can detect and fix them, they become irreversible logical errors. This is captured by the Erlang-A system as abandonment.

By analyzing the impact the impulses have on the system and the steady state behavior of the system, we are able to optimize the timing of impulse application under the periodic error correction scheme, and reduce average error in the system without introducing extra resource overhead.

This paper presents a comparative analysis of two prominent impulsive differential equation models frequently employed in queueing theory. The first model, the infinite server model, serves as a foundational framework. Characterized by a linear impulsive differential equation, it provides a simplified representation of queue dynamics, enabling the establishment of baseline behaviors. In contrast, the second model, the Erlang-A model, incorporates greater complexity. This model, described by a piecewise linear impulsive differential equation with two distinct regimes, reflects the intricacies of real-world scenarios. It accounts for a finite number of servers and the possibility of customer abandonment, acknowledging the limitations of resources and the impact of service delays on customer patience.

The primary objective of this work is to characterize the steady-state behavior of both models, with a specific focus on the oscillatory patterns exhibited by queue length processes. We conduct a comprehensive investigation into the amplitude of these oscillations across various operational regimes, providing insights into the maximum and minimum queue lengths encountered in steady state. This analysis is crucial for understanding system dynamics within the fixed interval between impulses.

Our analysis reveals that the intra-interval dynamics exhibit a striking resemblance to the charging and discharging behavior of a capacitor in an electrical circuit, offering a novel perspective on the underlying mechanisms governing the queueing system. Additionally, we compute the average queue length over an interval, providing valuable insights into the overall system performance and its implications for resource utilization and customer satisfaction. By combining these findings, we aim to enhance the understanding of impulsive differential equations in modeling queueing systems and offer practical insights for optimizing performance in various applications.


 \subsection{Main Contributions of the Paper}

\noindent The main contributions of this work can be summarized as follows:
\begin{itemize}
    \item We derive closed-form expressions for the steady-state lower and upper bounds of queue lengths governed by impulsive dynamics, revealing the conditions under which oscillatory behavior stabilizes.
    \item We provide analytical solutions for the average queue length over a single impulse interval, offering insights into system performance between discrete interventions.
    \item We formulate and solve an inverse problem that identifies optimal impulse times under both linear and Erlang-A regimes, giving the solution for both the maximizer and the minimizer.
\end{itemize}


\subsection{Organization of the Paper}

The remainder of this paper is structured as follows. 
\begin{enumerate}
    \item In Section 2 we introduce the infinite-server queue under impulsive control and derives explicit expressions for its amplitude and average dynamics.
    \item Section 3 models the Erlang-A queueing model with impulses. Depending on different system load, multiple structural theorems are established.
    \item Section 4 formulates and solves the inverse problem of impulse design. Specifically, we solve the optimal time to apply a single impulse given a fixed time horizon. This section considers both the linear and Erlang-A settings, with a detailed regime-based analysis for the latter.
\end{enumerate}

\section{Infinite Server Queue}

In this section, we study the infinite server queue with the intention of understanding how much time do adjacent customers spend in the system together.  A similar type of analysis has been completed \citet{kang2021queueing, palomooverlap}.  Before we get into studying the overlap times, we review some concepts and results for the $M_t/G/\infty$ queueing model as this will be very helpful in the analysis that follows.  

\subsection{Review of the \texorpdfstring{$M_t/G/\infty$}{Mt/G/infinity} Queue}

In this section, we review the $M_t/G/\infty$ queueing model. The $M_t/G/\infty$ queue $Q(t)$ has a Poisson distribution with time varying mean $q(t)$.  As observed in \citet{eick1993mt, eick1993physics}, $q(t)$ has the following integral representation
\begin{eqnarray}
q(t) &=& \mathbb{E}[Q(t) ] = \int^{t}_{-\infty} \overline{G}(t-u) \lambda(u) du 
\end{eqnarray}
where $\lambda(u)$ is the time varying arrival rate and $S$ represents a service time with distribution $G$, $\overline{G} = 1 - G(t) = 
\mathbb{P}( S > t)$.

We find it also useful to compute the mean of the $M/M/\infty$ queue using the solution of a linear differential equation.  

\begin{lemma} \label{ode_soln}
Let $q(t)$ be the solution to the following differential equation
\begin{equation}
\shortdot{q} = \lambda - \mu q(t) 
\end{equation}
where $q(0) = q_0$.  Then, the solution for any value of $t$ is given by 
\begin{equation}
q(t) = q_0 e^{-\mu t} + \frac{\lambda}{\mu} \left( 1 - e^{-\mu t} \right) = \frac{\lambda}{\mu} + \left( q_0 - \frac{\lambda}{\mu} \right) e^{-\mu t} 
\end{equation}
and as $t \to \infty$ we have
\begin{equation}
\lim_{t \to \infty} q(t) = \frac{\lambda}{\mu} .
\end{equation}
\begin{proof}
This follows from standard results on ordinary differential equations  (see \citet{tenenbaum1985ordinary}).  
\end{proof}
\end{lemma}

\begin{lemma} \label{ode_soln}
Let $q(t)$ be the solution to the following differential equation
\begin{equation}
\shortdot{q} = \lambda - \mu q(t) 
\end{equation}
where $q(0) = q_0$.  Then, the solution for any value of $t$ is given by 
\begin{equation}
q(t) = q_0 e^{-\mu t} + \frac{\lambda}{\mu} \left( 1 - e^{-\mu t} \right) = \frac{\lambda}{\mu} + \left( q_0 - \frac{\lambda}{\mu} \right) e^{-\mu t} 
\end{equation}
and as $t \to \infty$ we have
\begin{equation}
\lim_{t \to \infty} q(t) = \frac{\lambda}{\mu} .
\end{equation}
\begin{proof}
This follows from standard results on ordinary differential equations  (see \citet{tenenbaum1985ordinary}).  
\end{proof}
\end{lemma}

Now consider the impulsive differential equation:
Consider the impulsive differential equation:
\begin{eqnarray} \label{lineq1}
    \shortdot{q}(t) &=& \lambda - \mu q(t), \quad t \neq t_k, \\ \nonumber
    \Delta q|_{t = t_k} &=& q(t_k^+) - q(t_k^-) = b q(t_k^-), \quad t_k = k\delta, \quad k \in \mathbb{N},
\end{eqnarray}
where $\lambda$ is the arrival rate, $\mu$ is the service rate, $b$ is the proportional jump magnitude, $\delta = t_k - t_{k-1}$ is the time interval between impulses.

With the linear impulsive differential equation now established, we present a key result that characterizes the limiting amplitude of the system in terms of its model parameters. This result provides crucial insights into the long-term behavior of the equation under the specified conditions. Throughout this section, we assume the stability criterion $(1+b)e^{-\mu \delta} < 1$, which ensures exponential stability of the system. This assumption plays a central role in constraining the dynamics and guaranteeing that the oscillations in the solution decay over time, leading to a stable steady-state amplitude. 

\begin{theorem} \label{upperlowerinf}
    Let $L,U$ be the lower and upper values of the linear impulsive differential equation in steady state and let $-1< b< 0 $ and $(1+b)e^{-\mu \delta} < 1$, then 
    \begin{eqnarray}
     L &=&  \frac{\lambda (1+b)}{\mu} \left( \frac{1 - e^{-\mu \delta}}{1 - (1+b)e^{-\mu \delta}} \right) \\
    U &=& \frac{\lambda}{\mu} \left( \frac{1 - e^{-\mu \delta}}{1 - (1+b)e^{-\mu \delta}} \right).
    \end{eqnarray}
    Moreover, if $ 0 \leq b < e^{\mu \delta} -1 $, then L and U are switched.  
\begin{proof}
    In the steady state, the time right before the impulse reaches a maximum we denote as $U$ i.e. $U = q(t_k^-)$ for large enough $k$.  The time right after the impulse we denote as $L$ i.e. $L =q(t_k^+)$ for large enough $k$.  From our definition of the impulsive differential equation, we have 
    \begin{equation}
        \Delta q|_{t = t_k} = q(t_k^+) - q(t_k^-) = b q(t_k^-),
    \end{equation}
    which implies that 
    \begin{equation}
       q(t_k^+) = (1+b) q(t_k^-).
    \end{equation}
    Thus by substituting $L =q(t_k^+)$ and $U = q(t_k^-)$ , we have the following equations 
    \begin{eqnarray} \label{Leqn}
        L &=& (1+b) U .
    \end{eqnarray}
    Moreover, from Lemma \ref{ode_soln} we have
    \begin{eqnarray} \label{Ueqn}
        U &=& \frac{\lambda}{\mu} + \left( L - \frac{\lambda}{\mu} \right) e^{-\mu \delta} .
    \end{eqnarray}
    Using these two equations, we have
        \begin{eqnarray}
        U &=& \frac{\lambda}{\mu} + \left( L - \frac{\lambda}{\mu} \right) e^{-\mu \delta} \\
        &=& \frac{\lambda}{\mu} + \left( (1+b) U  - \frac{\lambda}{\mu} \right) e^{-\mu \delta} \\
        &=& \frac{\lambda}{\mu} \left( \frac{1 - e^{-\mu \delta}}{1 - (1+b)e^{-\mu \delta}} \right).
    \end{eqnarray}
    Since we have proved the result for $U$, it follows from Equation \ref{Leqn} that the result is also proved for $L$.  For the case when $ 0 < b < e^{\mu \delta} -1 $, the results are switched since the impulse causes the system to increase and not decrease, thus $U =q(t_k^+)$ and $L = q(t_k^-)$. This completes the proof. 
\end{proof}
\end{theorem}

Figure \ref{Figure_1} illustrates the behavior of the impulsive differential equation across four distinct regimes. These regimes are categorized by the direction of the impulse (increasing or decreasing) and the initial condition relative to the steady state (below or above). Regime 1 represents a decreasing impulse with an initial condition below the steady state. Regime 2 depicts a decreasing impulse with an initial condition above the steady state. Regime 3 showcases an increasing impulse with an initial condition below the steady state. Finally, Regime 4 illustrates an increasing impulse with an initial condition above the steady state. We observe that in all regimes considered, Theorem \ref{upperlowerinf} thoroughly encapsulates the amplitude dynamics of the linear impulsive differential equation, providing an exceptionally accurate and detailed representation of its dynamic and oscillatory behavior under diverse parameter settings.

 \begin{figure}[!htbp]
\centering
\subfloat[]{\includegraphics[scale=.35]{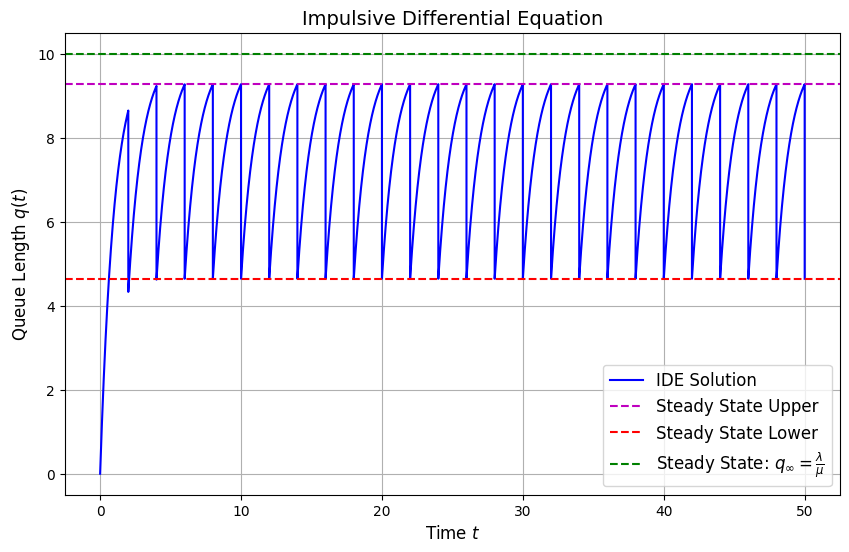}}
\subfloat[]{\includegraphics[scale=.35]{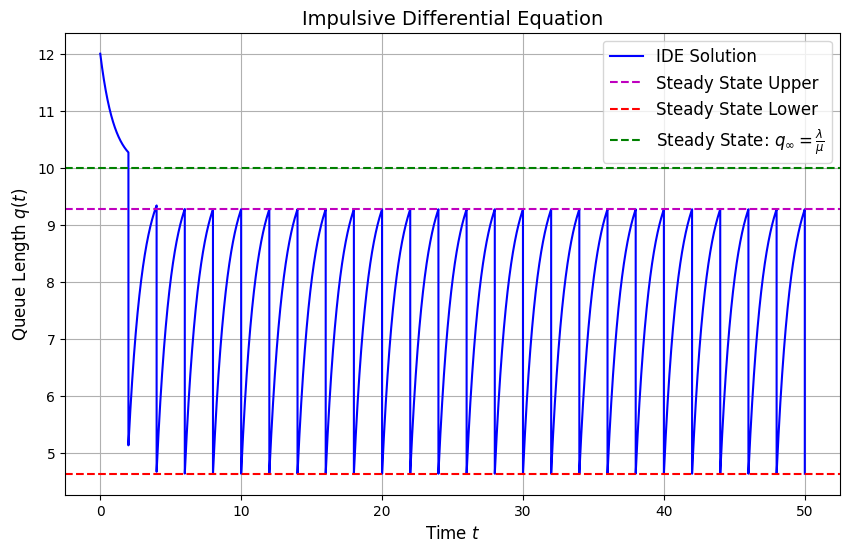}}
\\
\subfloat[]{\includegraphics[scale=.35]{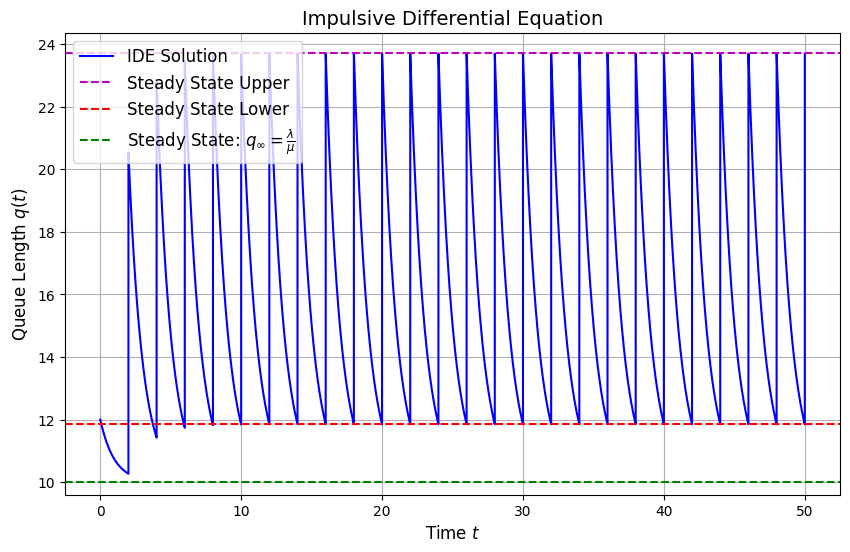}}
\subfloat[]{\includegraphics[scale=.35]{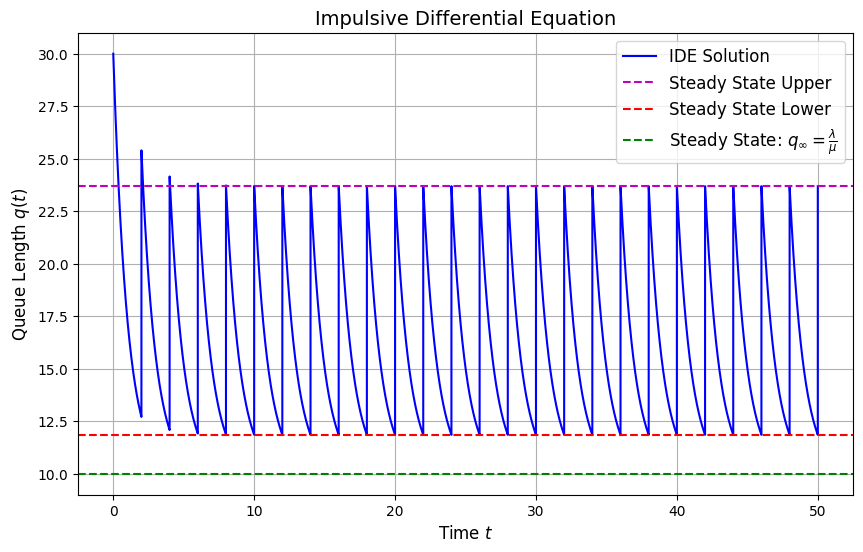}}
 \caption{Linear Impulsive Differential Equation Plots. \textbf{(a)} \ $\lambda = 10, \mu = 1, b = -1/2, q_0 = 0, \delta = 2$. \textbf{(b)}   \ $\lambda = 10, \mu = 1, b = -1/2, q_0 = 12, \delta = 2$. \textbf{(c)}  \ $\lambda = 10, \mu = 1, b = 1, q_0 = 12, \delta = 2$ . \textbf{(d)}  \ $\lambda = 10, \mu = 1, b = 1, q_0 = 30, \delta = 2$. }
\label{Figure_1}
\end{figure}

Note that the amplitude of the oscillatory dynamics is equal to 
\begin{eqnarray}
    \text{Amplitude} &=& U-L = (-1)^{(\sign(b) -1)} b \frac{\lambda}{\mu} \left( \frac{1 - e^{-\mu \delta}}{1 - (1+b)e^{-\mu \delta}} \right)
\end{eqnarray}Now that we have described the amplitude dynamics, we also are interested in the average queue length during an interval.  We know that it is not simply the average of $L$ and $U$.  We need to integrate the queue length over this interval of time $\delta$.

\begin{theorem}
The average queue length, $\mathcal{A} $, during a $\delta$ time interval in steady state is equal to 
\begin{eqnarray}
    \mathcal{A} = \frac{\lambda}{\mu} + \left( \frac{\lambda (1+b)}{\mu^2 \delta} \left( \frac{1 - e^{-\mu \delta}}{1 - (1+b)e^{-\mu \delta}} \right) - \frac{\lambda}{\mu^2 \delta}\right).
\end{eqnarray}
\begin{proof}
We leverage that the average is just the integral of queue length during one interval of size $\delta$.
    \begin{eqnarray}
        \mathcal{A}  &=& \frac{1}{\delta} \int^{\delta}_0 q(s) ds \\
        &=& \frac{1}{\delta} \int^{\delta}_0 \left( \left( q(0) - \frac{\lambda}{\mu} \right) e^{-\mu s} + \frac{\lambda}{\mu} \right) ds \\
        &=& \frac{1}{\delta} \int^{\delta}_0 \left( \left( L - \frac{\lambda}{\mu} \right) e^{-\mu s} + \frac{\lambda}{\mu} \right) ds \\
        &=& \frac{1}{\mu \delta} \left( L - \frac{\lambda}{\mu} \right) \left( 1 - e^{-\mu \delta} \right) + \frac{\lambda}{\mu}   \\
        &=& \frac{\lambda}{\mu} + \left( \frac{\lambda (1+b)}{\mu^2 \delta} \left( \frac{1 - e^{-\mu \delta}}{1 - (1+b)e^{-\mu \delta}} \right) - \frac{\lambda}{\mu^2 \delta}\right).
    \end{eqnarray}
\end{proof}
\end{theorem}

\section{The Erlang-A Queueing Model} \label{secQMod}

The Erlang-A queue, also known as the $M_{\lambda}/M_{\mu}/c+M_{\theta}$ queue is a stochastic model describing a queueing system with Poisson arrivals, exponential service times, and customer abandonment. It is defined by the following characteristics:

\begin{itemize}
    \item \textbf{Arrival Process}: Customers arrive according to a Poisson process with rate \( \lambda \), so the interarrival times are exponentially distributed with mean $1/\lambda$.
    \item \textbf{Service Process}: The system has \( s \) servers, and the service times follow an exponential distribution with mean \( 1/\mu \) (rate \( \mu \)).
    \item \textbf{Abandonment Process}: Customers waiting in the queue have a patience time that is exponentially distributed with mean \( 1/\theta \) (rate \( \theta \)). If a customer's waiting time exceeds their patience time, they abandon the queue.
\end{itemize}

The Erlang-A queueing model and its variants are extensively studied in the queueing literature, see for example \citet{zeltyn2005call, whitt2006sensitivity, mandelbaum2007service, gurvich2014excursion,  engblom2014approximations,  pender2014gram, niyirora2016optimal, aktekin2016stochastic, braverman2017stein, pender2017sampling, pender2017approximating, bitton2019joint, azriel2019erlang, van2019economies}. The Erlang-A queue occupies a prominent position within queueing theory due to its encompassing nature. Three foundational queueing models, namely the M/M/$\infty$ , M/M/c/c, and M/M/c/$\infty$  queues, can be derived as special cases of the Erlang-A model.  

The M/M/$\infty$ queue, characterized by an infinite server capacity, can be obtained from the Erlang-A model through two distinct approaches. Firstly, by setting the number of servers to infinity ($c = \infty$ ), the Erlang-A model degenerates into the M/M/$\infty$  queue. Alternatively, when the abandonment rate ($\theta$) is equated to the service rate $\mu$, the Erlang-A model also reduces to the M/M/$\infty$  queue. 

The M/M/c/c queue, a loss system where arriving customers are blocked if all servers are occupied, emerges from the Erlang-A model as the abandonment rate ($\theta$) approaches infinity ($\theta \to \infty$ ). This phenomenon, known as blocking, was first formally observed by \citet{hampshire2009time}).  

Finally, the M/M/c/$\infty$  queue, a system with a finite number of servers and no customer abandonment, is directly obtained by setting the abandonment rate to zero ($\theta =0$). This is intuitively evident, as the absence of an abandonment process is a defining characteristic of the M/M/c/$\infty$ queue.  This hierarchical relationship between the Erlang-A model and these classic queueing models underscores its significance as a unifying framework for understanding a wide range of queueing phenomena.

In \citet{halfin1981heavy}, an important insight emerged regarding multi-server queueing systems, emphasizing the ability to scale up both the arrival rate and the number of servers simultaneously. This scaling approach, known as the \emph{Halfin-Whitt} scaling, has become pivotal in modeling call centers within queueing literature, as illustrated in works such as \citet{pang2007martingale}. Shortly thereafter \citet{mandelbaum1998strong, massey2018dynamic}, showed that a scaled queue length process converges almost surely to a deterministic differential equation i.e.
 \begin{equation}\label{eq:fluid-limit}
\lim_{\eta\to\infty} \frac{1}{\eta} Q^{\eta}(t) =   q(t) \hspace{3mm}
\mathrm{a.s.}
\end{equation}
where the deterministic process $q(t)$, the \emph{fluid mean}, satisfies the following one dimensional ordinary differential equation (ODE),
\begin{equation} 
\label{fldmean}
 \shortdot{q}(t) = \lambda(t) - \mu \cdot (q(t) \wedge c) - \theta \cdot (q(t) - c)^+ .
\end{equation}

More recently, \citet{ko2023number} shows that the steady state dynamics of Equation (\ref{fldmean}) are given by
\begin{eqnarray*}
\lim_{t \to \infty} q(t) & = & \begin{cases}
\frac{\lambda-\mu c+\theta c}{\theta} & \mbox{if }\lambda>\mu c\\
\frac{\lambda}{\mu} & \mbox{if }\lambda\leq\mu c.
\end{cases}
\end{eqnarray*}

Now this section we consider the following impulsive differential equation:

\begin{eqnarray} \label{abeqn1}
    \shortdot{q}(t) &=& \lambda - \mu \left( q(t) \wedge c \right) - \theta ( q(t) - c )^+, \quad t \neq t_k, \\ \nonumber
    \Delta q|_{t = t_k} &=& q(t_k^+) - q(t_k^-) = b q(t_k^-), \quad t_k = k\delta, \quad k \in \mathbb{N},
\end{eqnarray}
where $\lambda$ is the arrival rate, $\mu$ is the service rate, $b$ is the proportional jump magnitude, $\delta$ is the time interval between impulses and $\theta$ is the abandonment rate.  Unlike the linear impulsive differential equation given in Equation (\ref{lineq1}), this equation involves many cases since the equation has two distinct dynamics above and below the number of servers depending on the model parameters.

\begin{theorem} \label{case1}
   Suppose we have impulsive differential equation given in Equation (\ref{abeqn1}) and we have the two conditions $\lambda < \mu c$ and $-1< b< 0 $.  Then, if $L,U$ denote the lower and upper values of the impulsive differential equation in steady state, we have 
    \begin{eqnarray}
     L &=&  \frac{\lambda (1+b)}{\mu} \left( \frac{1 - e^{-\mu \delta}}{1 - (1+b)e^{-\mu \delta}} \right) \\
    U &=& \frac{\lambda}{\mu} \left( \frac{1 - e^{-\mu \delta}}{1 - (1+b)e^{-\mu \delta}} \right).
    \end{eqnarray}
    Moreover, if $ 0 \leq b < e^{\mu \delta} -1 $, then L and U are switched.  
\begin{proof}
The proof is identical to the linear case. This completes the proof. 
\end{proof}
\end{theorem}

\begin{theorem}
Suppose we have impulsive differential equation given in Equation (\ref{abeqn1}) and we have the two conditions $\lambda \geq \mu c$ and $-1< b< 0 $.  Then, if $L,U$ denote the lower and upper values of the impulsive differential equation in steady state, we have 
    \begin{eqnarray}
     L &=&   ( 1+ b) \left( c + \frac{\lambda-\mu c}{\theta} \right) \left( \frac{1 - e^{-\mu \delta}}{1 - (1+b)e^{-\mu \delta}} \right)\\
    U &=& \left( c + \frac{\lambda-\mu c}{\theta} \right) \left( \frac{1 - e^{-\mu \delta}}{1 - (1+b)e^{-\mu \delta}} \right).
    \end{eqnarray}
\begin{proof}
    In the steady state, the time right before the impulse reaches a maximum we denote as $U$.  The time right after the impulse we denote as $L$.  Thus, we have the following equations 
    \begin{eqnarray} \label{Leqnab1}
        L &=& (1+b) U 
    \end{eqnarray}
    and 
    \begin{eqnarray} \label{Ueqnab1}
        U &=& \frac{\lambda}{\mu} + \left( L - \frac{\lambda}{\mu} \right) e^{-\mu \delta} .
    \end{eqnarray}
    Using these two equations, we have
        \begin{eqnarray}
        U &=& c + \frac{\lambda-\mu c}{\theta} + \left( L - c - \frac{\lambda-\mu c}{\theta} \right) e^{-\theta \delta} \\
        &=& c + \frac{\lambda-\mu c}{\theta} + \left( (1+b) U - c - \frac{\lambda-\mu c}{\theta} \right) e^{-\theta \delta} \\
        &=& \left( c + \frac{\lambda-\mu c}{\theta} \right) \left( \frac{1 - e^{-\theta \delta}}{1 - (1+b)e^{-\theta \delta}} \right).
    \end{eqnarray}
    The equality for $L$ follows from the relationship of $L$ and $U$ in Equation (\ref{Leqn}).  This completes the proof. 
\end{proof}
   C\end{theorem}

A new theorem where we have the impulse only work on the abandonment folks

\begin{theorem}
    Let $L,U$ be the lower and upper values of the linear impulsive differential equation in steady state, then 
    \begin{eqnarray}
     L &=&   ( 1+ b) \left( c + \frac{\lambda-\mu c}{\theta} \right) \left( \frac{1 - e^{-\mu \delta}}{1 - (1+b)e^{-\mu \delta}} \right)\\
    U &=& \left( c + \frac{\lambda-\mu c}{\theta} \right) \left( \frac{1 - e^{-\mu \delta}}{1 - (1+b)e^{-\mu \delta}} \right).
    \end{eqnarray}
\begin{proof}
    In the steady state, the time right before the impulse reaches a maximum we denote as $U$.  The time right after the impulse we denote as $L$.  Thus, we have the following equations 
    \begin{eqnarray} \label{Leqnab}
        L &=&  (U \wedge c) + (1+b) ( U - c)^+ 
    \end{eqnarray}
    and 
    \begin{eqnarray} \label{Ueqnab}
        U &=& \frac{\lambda}{\mu} + \left( L - \frac{\lambda}{\mu} \right) e^{-\mu \delta} .
    \end{eqnarray}
    Using these two equations, we have
        \begin{eqnarray}
        U &=& c + \frac{\lambda-\mu c}{\theta} + \left( L - c - \frac{\lambda-\mu c}{\theta} \right) e^{-\theta \delta} \\
        &=& c + \frac{\lambda-\mu c}{\theta} + \left( (U \wedge c) + (1+b) ( U - c)^+  - c - \frac{\lambda-\mu c}{\theta} \right) e^{-\theta \delta} \\
        &=& c + \frac{\lambda-\mu c}{\theta} + \left( (1+b) ( U - c)^+   - \frac{\lambda-\mu c}{\theta} \right) e^{-\theta \delta} \\
        &=&  c + \frac{\lambda-\mu c}{\theta} + \left( (1+b)  U - (1+b) c   - \frac{\lambda-\mu c}{\theta} \right) e^{-\theta \delta} \\
        &=& \left( c + \frac{\lambda-\mu c}{\theta} \right) \left( \frac{1 - e^{-\theta \delta}}{1 - (1+b)e^{-\theta \delta}} \right) - \frac{cb e^{-\theta }}{1 - (1+b)e^{-\theta \delta}}
    \end{eqnarray}
    The equality for $L$ follows from the relationship of $L$ and $U$ in Equation \ref{Leqn}.  This completes the proof. 
\end{proof}
\end{theorem}


\section{Designing the Impulse}

In impulsive differential equations, the state of the system undergoes sudden changes—often called "jumps"—at specific moments in time, determined by either fixed schedules or dynamic criteria. Designing when these impulses occur is crucial because the timing of interventions can dramatically alter the system's trajectory. If the impulses are too frequent or too rare, the system may diverge from desired behavior or become unstable. Therefore, determining optimal impulse times is a key part of system design, allowing for better control, stability, or performance. This design aspect enables the integration of discrete events into continuous dynamics, providing a powerful hybrid modeling framework that closely mirrors real-world systems where abrupt changes happen by necessity or design.

Applications of carefully designed impulsive interventions are widespread. In medicine, for example, timed drug administration in pharmacokinetics—such as insulin injections for diabetics—can be modeled using impulsive systems, where the effect of a dose causes an instantaneous drop in blood glucose. In ecology, pest control strategies often rely on scheduled impulses (e.g., pesticide applications) to reduce population levels at critical times. In engineering and robotics, impulse control is used in spacecraft navigation where thrusters provide sudden bursts of force to correct trajectories. Even in economics, impulse models are used to describe the impact of sudden policy changes or market shocks. In all these contexts, the value lies in precisely determining when to act, making the timing of impulses not just a modeling detail, but a central design problem.

In addition to medicine, ecology, and engineering, impulsive differential equations find important applications in neural modeling, where neurons exhibit sudden spikes in membrane potential known as action potentials. These spikes occur at discrete times and can be modeled as impulses, with the timing of each spike playing a critical role in how information is transmitted and processed in the brain. In power systems, impulse control is used to model and mitigate faults or switching events in electrical networks, helping ensure system stability and avoiding cascading failures. In supply chain management, inventory restocking can be viewed as an impulsive process, where stock is replenished at discrete times to avoid shortages or overstocking—optimally timing these interventions improves efficiency and reduces costs. Across these diverse fields, the ability to model and design the timing of sharp interventions offers a strategic advantage in controlling complex systems with both continuous dynamics and discrete events.

Another compelling area of application for impulsive differential equations is epidemiology, where public health interventions—such as vaccination campaigns, lockdowns, or mass testing—occur at discrete times to curb disease spread. Strategically timing these impulses can significantly influence the outbreak dynamics and reduce transmission. In financial mathematics, impulsive models capture sudden market interventions, such as interest rate adjustments or regulatory actions, where the market state changes abruptly. Similarly, in manufacturing systems, maintenance or quality control inspections can be modeled as impulses that reset or adjust the system state at scheduled intervals, optimizing productivity and reducing downtime. These examples illustrate how impulse timing is not only a mathematical feature but also a policy lever with substantial real-world impact, allowing decision-makers to guide complex systems toward desired outcomes.

\subsection{Impulse Design in Infinite Server}
While the previous sections focused on characterizing the amplitude and average dynamics of impulsive differential equations with fixed jump intervals, we now turn our attention to a practically significant inverse problem. More specifically, given a fixed time interval $[0,T]$, we seek to determine the optimal time to apply a single impulse in order to minimize/maximize average queue length. 

To address this optimization challenge systematically, we examine in this section the linear random impulsive differential equation (RIDE), which serves as a fundamental model for understanding the relationship between impulse timing and system performance. The linear framework provides analytical tractability while capturing the essential dynamics that govern the timing-performance relationship in more complex systems.

\textbf{}Now consider the impulsive differential equation:
Consider the impulsive differential equation:
\begin{eqnarray} \label{lineq1}
    \shortdot{q}(t) &=& \lambda - \mu q(t), \quad t \neq t_k, \\ \nonumber
    \Delta q|_{t = t_k} &=& q(t_k^+) - q(t_k^-) = b q(t_k^-), \quad t_k = k\delta, \quad k \in \mathbb{N},
\end{eqnarray}
where $\lambda$ is the arrival rate, $\mu$ is the service rate, $b$ is the proportional jump magnitude, $\delta = t_k - t_{k-1}$ is the time interval between impulses.  

We know that if $q(t)$ is the solution to the following differential equation
\begin{equation}
\shortdot{q} = \lambda - \mu q(t) 
\end{equation}
where $q(0) = q_0$, then, the solution for any value of $t$ is given by 
\begin{equation}
q(t) = \frac{\lambda}{\mu} + \left( q_0 - \frac{\lambda}{\mu} \right) e^{-\mu t} .
\end{equation}

To study the random case, we can use the method of iteration to see that the first time an impulse happens is 
\begin{eqnarray}
q(t_1) &=& \frac{\lambda}{\mu} + \left( q_0 - \frac{\lambda}{\mu} \right) e^{-\mu t_1} \label{eqn: basecase1}
\end{eqnarray}
and
\begin{eqnarray}
q(t_2) &=& \frac{\lambda}{\mu} + \left( (1+b)q(t_1) - \frac{\lambda}{\mu} \right) e^{-\mu (t_2- t_1)} \\
&=& \frac{\lambda}{\mu} + \left( \frac{\lambda (1+b)}{\mu} + (1+b)\left( q_0 - \frac{\lambda}{\mu} \right) e^{-\mu t_1}  - \frac{\lambda}{\mu} \right) e^{-\mu(t_2 - t_1)} \\
&=& \frac{\lambda}{\mu} + \left(  \frac{\lambda b}{\mu} + (1+b) \left( q_0 - \frac{\lambda}{\mu} \right) e^{-\mu t_1}  \right) e^{-\mu(t_2 - t_1)} \\
&=& \frac{\lambda}{\mu} +  \frac{\lambda b}{\mu} e^{-\mu(t_2 - t_1)}  + (1+b) \left( q_0 - \frac{\lambda}{\mu} \right) e^{-\mu t_2}  
\end{eqnarray}

\begin{theorem}
    Given that exactly 1 impulse is applied to the system at time $\tau$ during time interval $[0,T]$, the optimal time $\tau_U^*$ to apply the impulse so that the average queue length is maximized is:
    $$\tau_U^* = T,$$
    whereas the optimal time $\tau_L^*$ to apply the impulse so that the average queue length is minimized is:
    $$\tau_L^* = \begin{cases}
        \frac{T}{2} + \frac{\ln \left( 1-q_0\frac{\mu}{\lambda} \right)}{2 \mu} & \text{when } q_0 < \frac{\lambda}{\mu}\\
        0 & \text{o.w.}
    \end{cases} $$

    \begin{proof}
    We first let $\tau$ be the time of the impulse and define the average queue length and the optimizers more rigorously in a mathematical manner:
    $$J(\tau) = \frac{1}{T} \left( \int_0 ^\tau q_{1}(t) dt +  \int_\tau^T q_{2}(t) dt \right)$$
    $$\tau_U^* = \underset{\tau}{\mathrm{argmax}}_{\tau \in [0,T]} J(\tau) , \quad \quad \tau_L^* = \underset{\tau}{\mathrm{argmin}}_{\tau \in [0,T]} J(\tau).$$
    where $q_{1}(t)$ and $q_1(t)$ denote the solution for our differential Equation \ref{lineq1} on intervals $[0,\tau]$ and $[\tau, T]$.
        Denoting the stationary mean $\xi = \frac{\lambda}{\mu}$ and write down the solution for our differential equation:

\begin{align*}
    q_1(t) &= \xi + (q_0 - \xi) e^{-\mu t}  &\text{ for } 0 \leq t \leq \tau \\
    q_2(t) &= \xi + \bigg( b\xi + b(q_0-\xi)e^{-\mu \tau} - \xi \bigg) e^{-\mu(t-\tau)} \\
    &= \xi_2 + (b-1) \xi_2 \cdot e^{-\mu(t-t_1)} + b(q_0-\xi_2)e^{-\mu t}  &\text{ for } \tau < t \leq T
\end{align*}

Thus,
\begin{align*}
    \int_0^T q(t) dt &= \int_0^\tau q_1(t) dt + \int_\tau^T q_2(t)\\
    &= \xi t + \frac{q_0 - \xi}{\mu}\left(1-e^{-\mu \tau}\right) + \frac{\lambda(b-1)}{\mu^2}\left(1-e^{-\mu (T-\tau)}\right) + \frac{b(q_0 - \xi)}{\mu}\left(e^{-\mu \tau}-e^{-\mu T}\right)\\
    \frac{d\left(\int_0^T q(t) dt\right)}{d\tau} &=( {q_0 - \xi})e^{-\mu \tau}-\frac{\lambda(b-1)}{\mu}e^{-\mu (T-\tau)}-b( {q_0 - \xi})e^{-\mu \tau}\\
    &= (1-b)( {q_0 - \xi})e^{-\mu \tau}-\frac{\lambda(b-1)}{\mu}e^{-\mu (T-\tau)}
\end{align*}

We notice:

(i) When $q_0 > \xi$, the derivative function is strictly positive, making 
$$\tau_L^* = \underset{\tau}{\mathrm{argmin}}_{\tau} \int_0^T q(t) dt = 0,$$
$$\tau_U^* =\underset{\tau}{\mathrm{argmax}}_{\tau} \int_0^T q(t) dt = T.$$

(ii) When $q_0 \leq \xi$, set derivative to 0:
\begin{align*}
    ( \xi-q_0)e^{-\mu \tau_L^*}=\xi e^{-\mu (T-\tau_L^*)}\\
    \ln(\xi-q_0)-\mu\tau_L^* = \ln(\xi)-\mu T +\mu\tau_L^*\\
    \tau_L^* = \frac{\ln\left(\frac{\xi-q_0}{\xi}\right)}{2\mu} + \frac{T}{2}.
\end{align*}
And since $q(t) = \xi + (q_0-\xi)e^{-\mu t}$ when $\tau = T$ and $q(t) = \xi + (bq_0-\xi)e^{-\mu t}$ when $\tau = 0$, we have
$$\tau_U^* = \underset{\tau}{\mathrm{argmax}}_{t_1} \int_0^t q(t) dt = T.$$

    \end{proof}
\end{theorem}

\subsection{Impulse Design in Erlang-A}

With the linear RIDE framework in the preceding section shedding some insights into optimal impulse timing, we now extend our analysis to a more realistic and practically relevant queueing system. More specifically, we address the same inverse optimization problem but now considering the underlying dynamics of an Erlang-A queue.

The Erlang-A queueing model represents a significant advancement in complexity from the linear case, as it incorporates the stochastic nature of customer arrivals, service processes, and importantly, customer abandonment behavior. Unlike the linear dynamics studied previously, impulse optimization in the Erlang-A system is more complicated, lacking general analytical solution sometimes due to the interaction between queue length, service capacity, and abandonment rates.

This extension to the Erlang-A framework presents analytical challenges that were absent in the linear case. The presence of capacity and customer abandonment introduces complexity as the system's response to impulses now depends on these parameters and also where the queue started.
In this section, we investigate how the optimal impulse timing strategies derived for linear systems translate to this more complex queueing environment, and identify the key factors that distinguish optimal control policies in different scenarios.

Formally, in this section, we continue to assume that the system evolves according to the fluid-limit impulsive differential equation described by Equation \ref{fldmean} with impulse applied according to Equation \ref{abeqn1}. We alter the definition of $b$ in Equation \ref{abeqn1} slightly to:
\begin{eqnarray} \label{eqn: erlangdefinition}
    \shortdot{q}(t) &=& \lambda - \mu \left( q(t) \wedge c \right) - \theta ( q(t) - c )^+, \quad t \neq t_k, \\ \nonumber
    \Delta q|_{t = t_k} &=& q(t_k^+) - q(t_k^-) = (b-1) q(t_k^-), \quad t_k = k\delta, \quad k \in \mathbb{N},
\end{eqnarray}
for the sake of simplicity in calculation. The objective is to select $\tau_U^* \in [0,T]$ and $\tau_L^* \in [0,T]$ that respectively maximizes and minimizes the average queue length:
$$J(\tau) = \frac{1}{T} \left( \int_0 ^\tau q_{1}(t) dt +  \int_\tau^T q_{2}(t) dt \right).$$
i.e.$$\tau_U^* = \underset{\tau}{\mathrm{argmax}}_{\tau \in [0,T]} J(\tau), \quad \quad \tau_L^* = \underset{\tau}{\mathrm{argmin}}_{\tau \in [0,T]} J(\tau).$$
In the above expressions, $q_{1}(t)$ and $q_1(t)$ denote the solution for our differential equation \ref{abeqn1} on intervals $[0,\tau]$ and $[\tau, T]$.
Before proceeding to the main result on optimal jumping times, we first present a result from recent research by \citet{ko2023number}, in which the explicit solution to the fluid queue length in the steady state is presented.

\begin{proposition} \label{prop: ko}
    When $\lambda(t)$ and $\mu(t)$ are constants, the solution $q(t)$ to Eq. (4) is given as follows:
    \begin{equation}
        q(t) = \begin{cases}
        \frac{\lambda-\mu c+\theta c}{\theta} + \left(q(0) - \frac{\lambda-\mu c+\theta c}{\theta}\right) e^{-\theta t} & \text{if } q(0) > c, \lambda > \mu c \\[0.5em]
        \frac{\lambda-\mu c+\theta c}{\theta} + \left(q(0) - \frac{\lambda-\mu c+\theta c}{\theta}\right) e^{-\theta t}, & \text{if } q(0) > c, \lambda \leq \mu c, t \leq t_1^* \\[0.5em]
        \frac{\lambda}{\mu} + \left(c - \frac{\lambda}{\mu}\right) e^{-\mu(t-t^*)}, & \text{if } q(0) > c, \lambda \leq \mu c, t > t_1^* \\[0.5em]
        \frac{\lambda}{\mu} + \left(q(0) - \frac{\lambda}{\mu}\right) e^{-\mu t} & \text{if } q(0) \leq c, \lambda \leq \mu c \\[0.5em]
        \frac{\lambda}{\mu} + \left(q(0) - \frac{\lambda}{\mu}\right) e^{-\mu t}, & \text{if } q(0) \leq c, \lambda > \mu c, t \leq t_2^* \\[0.5em]
        \frac{\lambda-\mu c+\theta c}{\theta} + \left(\frac{-\lambda+\mu c}{\theta}\right) e^{-\theta(t-t^*)}, & \text{if } q(0) \leq c, \lambda > \mu c, t > t_2^*
        \end{cases}
        \tag{6}
    \end{equation}
    where $t_1^* = \frac{\log\left(\frac{\theta q(0)-\lambda+\mu c-\theta c}{\mu c-\lambda}\right)}{\theta}$ and $t_2^* = \frac{\log\left(\frac{q(0)-\frac{\lambda}{\mu}}{c-\frac{\lambda}{\mu}}\right)}{\mu}$.
\end{proposition}

We can see that the dynamics of the system differs drastically depending on the initial queue size, and whether the system is overloaded ($\lambda > \mu c$), or underloaded ($\lambda \leq \mu c$). These distinctions give rise to different cases for analysis and also partition the time interval into distinct regimes according to impulse time. To tailor the notation for the analysis we provide in this paper, we modify the notation in Proposition \ref{prop: ko} slightly. Since $t_1^*$ and $t_2^*$ are essentially dependent on $q_0$, we use $t_{1,q_0}^*$ and $t_{2,q_0}^*$ to denote the time that the system reaches capacity $c$ given that it starts at $q_0$. Now we give the main result below, accounting for all instances and regimes.

\begin{theorem} \label{thm:erlang}
     Assuming an impulse is applied at time $\tau \in [0,T]$ to the Erlang-A system described by Equation \ref{eqn: erlangdefinition} with constant $\lambda(t), \mu(t)$ and initial condition $q(t) = q_0$. We also denote the steady state mean $\xi_1 = \frac{\lambda-\mu c + \theta c}{\theta}, \xi_2 = \frac{\lambda}{\mu}$ Then the optimal jumping time $\tau_U^*$ maximizing the average queue length $J(\tau)$ is equal to $T$.
    \begin{proof}
            The proof is provided in the A.ppendix
    \end{proof}
\end{theorem}

\begin{theorem} \label{thm:erlang}
    Assume an impulse is applied at time $\tau \in [0,T]$ to the Erlang-A system described by Equation \ref{eqn: erlangdefinition} with constant $\lambda(t), \mu(t)$ and initial condition $q(t) = q_0$. We also denote the steady state mean $\xi_1 = \frac{\lambda-\mu c + \theta c}{\theta}, \xi_2 = \frac{\lambda}{\mu}$ Then the structure of the minimizer $\tau_L^*$ depends on the initial condition $q_0$ and the relative load $\frac{\lambda}{\mu c}$, which falls into one of the following four cases: (i) $q_0 > c, \lambda > \mu c$, (ii) $q_0 \leq c, \lambda > \mu c$, (iii) $q_0 > c, \lambda \leq \mu c$, and (iv) $q_0 \leq c, \lambda \leq \mu c$. Under case $i$, the interval $[0,T]$ can be decomposed into sub-intervals $\{I_{i,j}\}_{j=1}^{n}$ depending on the case. The sub-intervals, and the optimizers whose closed form can be obtained on respective sub-intervals, denoted by $\tau_{I_{i,j}}^*$, are described in the follows.
        \begin{enumerate}[label=(\roman*)]
            \item  $I_{1,2-\mathbbm{1}\{q_0 > \xi_1\}} =\left[ 0, \frac{\ln\left( \frac{q_0-\xi_1}{c/b - \xi_1}\right)}{\theta} \right]$, $I_{1,1+\mathbbm{1}\{q_0 > \xi_1\}} = \left[ \frac{\ln\left( \frac{q_0-\xi_1}{c/b - \xi_1}\right)}{\theta}, T \right]$. 
            $$\tau_{I_{1,1}}^* = \begin{cases}
                \frac{T}{2} + \frac{\ln \left(1-q_0/\xi_1\right)}{2\theta} & \text{when } q_0 \leq \xi_1\\
        0 & \text{o.w.}
            \end{cases} ;$$
            
            \item When $T > \frac{\ln \left( \frac{b(q_0-\xi_1)}{c-b\xi_1} \right)}{\theta}$, $I_{2,1} = \left[ 0, t_{2,q_0}^*\right]$, $I_{2,2} = \left[t_{2,q_0}^*,\frac{\ln \left( \frac{b(q_0-\xi_1)}{c-b\xi_1} \right)}{\theta} \right]$, $I_{2,3} = \left[ \frac{\ln \left( \frac{b(q_0-\xi_1)}{c-b\xi_1} \right)}{\theta} , T\right]$, and
            $$\tau_{I_{2,1}} =t_{2,q_0}^* , \tau_{I_{2,3}} = \frac{T+t_{2,q_0}^*}{2} + \frac{\ln(1-c/\xi_2)}{2\theta};$$ When $t_{2,q_0}^*< T \leq \frac{\ln \left( \frac{b(q_0-\xi_1)}{c-b\xi_1} \right)}{\theta}$, $I_{2,1} = \left[ 0, t_{2,q_0}^*\right]$, $I_{2,4} = \left[t_{2,q_0}^*,T \right]$.
            
            When $T \leq t_{2,q_0}^*$, $I_{2,5} = \left[ 0,T \right]$, and
            $$ \tau_L^* =\tau_{I_{2,5}} = \frac{T}{2} + \frac{\ln \left( 1-q_0/\xi_2 \right)}{2\mu}.$$
            
            \item When $T > \frac{\ln\left( \frac{q_0-\xi_1}{c/b-\xi_1}\right)}{\theta}$, $I_{3,1} = \left[ 0, \frac{\ln \left( \frac{b(q_0-\xi_1)}{c-b\xi_1} \right)}{\theta} \right]$, $I_{3,2} = \left[ \frac{\ln \left( \frac{b(q_0-\xi_1)}{c-b\xi_1} \right)}{\theta}, t_{1,q_0}^* \right]$, $I_{3,3} = \left[ t_{1,q_0}^* , T\right]$, and
            $$\tau_{I_{3,1}} = 0, \tau_{I_{3,3}} = t_{1,q_0}^*;$$
            When $T \leq \frac{\ln\left( \frac{q_0-\xi_1}{c/b-\xi_1}\right)}{\theta}$, $I_{3,4} = \left[ 0, T\right] $, and 
            $$\tau_L^* =\tau_{I_{3,4}} = 0;$$

            \item $I_{4,1} = \left[ 0, T \right]$.
            $$\tau_L^* =\tau_{I_{4,1}}^* = \begin{cases}
        \frac{T}{2} + \frac{\ln(1-q_0/\xi_2)}{2 \mu} & \text{when } q_0 \leq \xi_2\\
        0 & \text{o.w.}
    \end{cases}$$
        \end{enumerate}
        On all remaining unspecified intervals, i.e. $I_{1,2}, I_{2,2}, I_{2,4}, I_{3,2} $, the optimizers lack analytic representations under general assumptions and should be minimized via numerical methods.

    \begin{proof}
            The proof could be found in the Appendix.
        \end{proof}
\end{theorem}

\section{Numerical Experiments}
To better illustrate the system dynamics and the optimal impulse time, we perform numerical experiment for each of the four scenarios specified in Theorem \ref{thm:erlang}, corresponding to figures 1-4. In each experiment, we solve for the ode numerically using the adaptive step-size Runge-Kutta method  with integration density 1000 and compute the average trajectory using the trapezoidal rule. Finally, we plot the average queue length as a function of time of the impulse. The red and yellow vertical lines represents the analytical solutions of optimal impulse application time with red line corresponding to the minimizer and yellow the maximizer. Note that on some intervals there is a lack of closed form analytical solution. For those intervals, we choose the modified Powell's hybrid method to numerically solve for the minimizer and synthesize with the minimizer on the other intervals. We choose this method to balance speed and robustness.

\subsection{\texorpdfstring{$q_0 > c, \lambda > \mu c$}{q0 > c, lambda > mu c}}
In the graphs below, Figure \ref{fig:1_a} corresponds to the case where the system starts below the steady state mean and Figure \ref{fig:1_b} corresponds to the case where the system starts above the steady state mean. The parameters used are shown in the plots.

\begin{figure}[h!]
    \centering
    \begin{subfigure}[b]{0.45\textwidth}
        \centering
        \includegraphics[width=\textwidth]{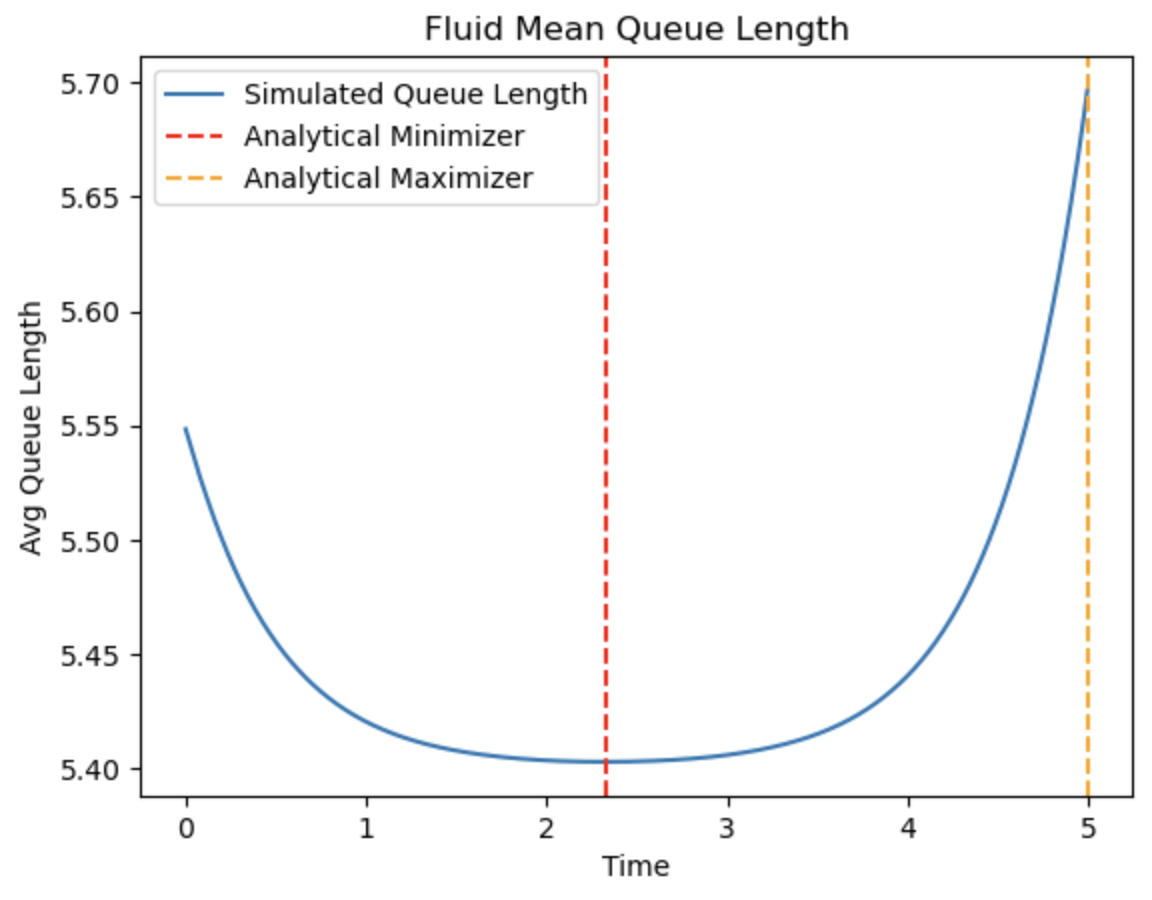}
        \caption{}
        \label{fig:1_a}
    \end{subfigure}
    \hfill
    \begin{subfigure}[b]{0.45\textwidth}
        \centering
        \includegraphics[width=\textwidth]{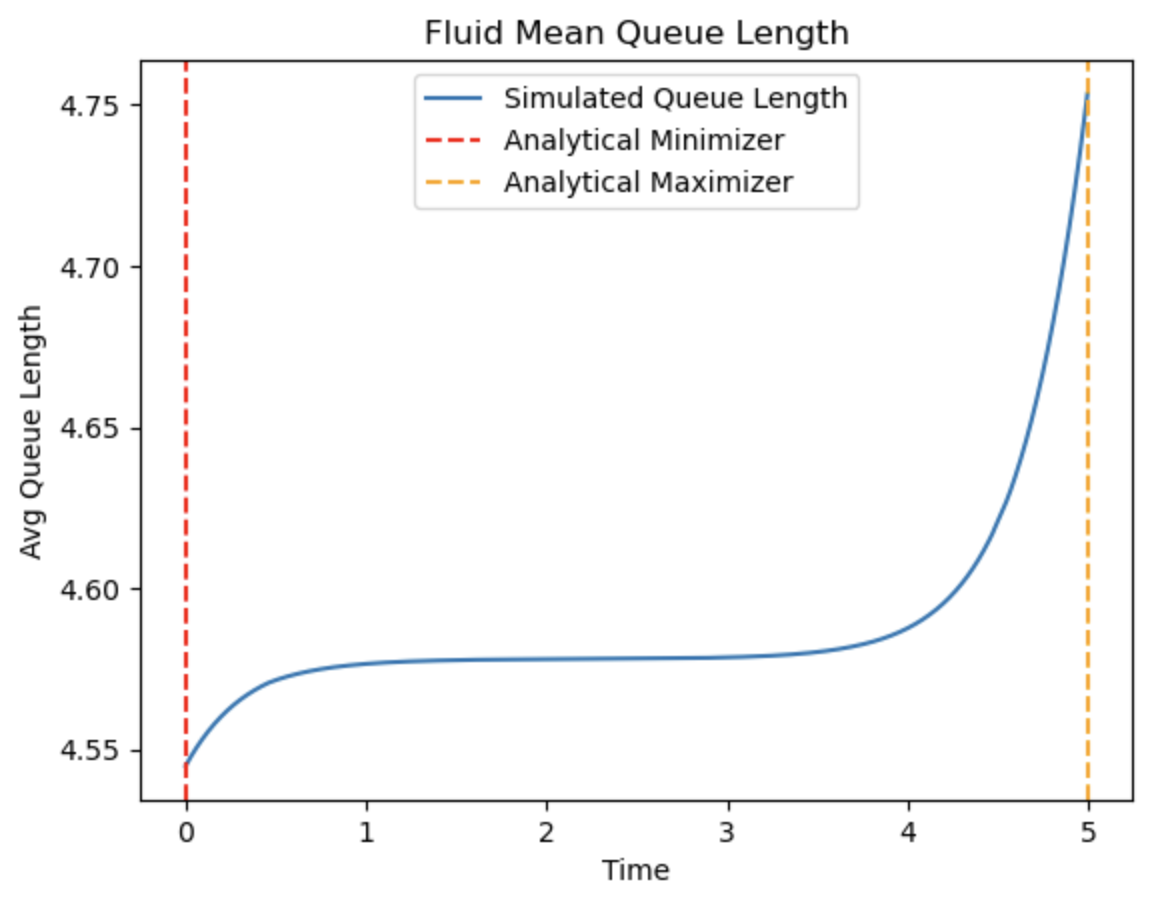}
        \caption{}
        \label{fig:1_b}
    \end{subfigure}
    \caption{Fluid Mean of Overload System (starts above capacity). \textbf{(a).} $\lambda = 10, \mu = 1, b = 0.5, T = 5, q_0 = 3, \theta = 2, c = 2$. \textbf{(b). }$\lambda = 10, \mu = 2, b = 0.5, T = 5, q_0 = 6, \theta = 3, c = 4$.}
    \label{fig:main_figure}
\end{figure}

\subsection{$q_0 \leq c, \lambda > \mu c$}
There are three sub-cases corresponding to different time horizons. We ran experiment of each of the sub-cases and plotted three plots correspondingly. Figure \ref{fig:2a} corresponds to the case where $T > \frac{\ln \left( \frac{b(q_0-\xi_2)}{c-b\xi_2} \right)}{\mu} $; Figure \ref{fig:2b} corresponds to the case where $t_{2,q_0}^*< T \leq \frac{\ln \left( \frac{b(q_0-\xi_2)}{c-b\xi_2} \right)}{\mu}$; and Figure \ref{fig:2c} corresponds to the case where $T \leq t_{2,q_0}^*$.  Under the parameters used for this case, which are specified in the plots, $\frac{\ln \left( \frac{b(q_0-\xi_2)}{c-b\xi_2} \right)}{\mu} \approx 0.405$ and $t_{2,q_0}^* \approx 0.223$.
\clearpage
\begin{figure}[h!]
    \centering
    \begin{subfigure}[b]{0.45\textwidth}
        \centering
        \includegraphics[width=\textwidth]{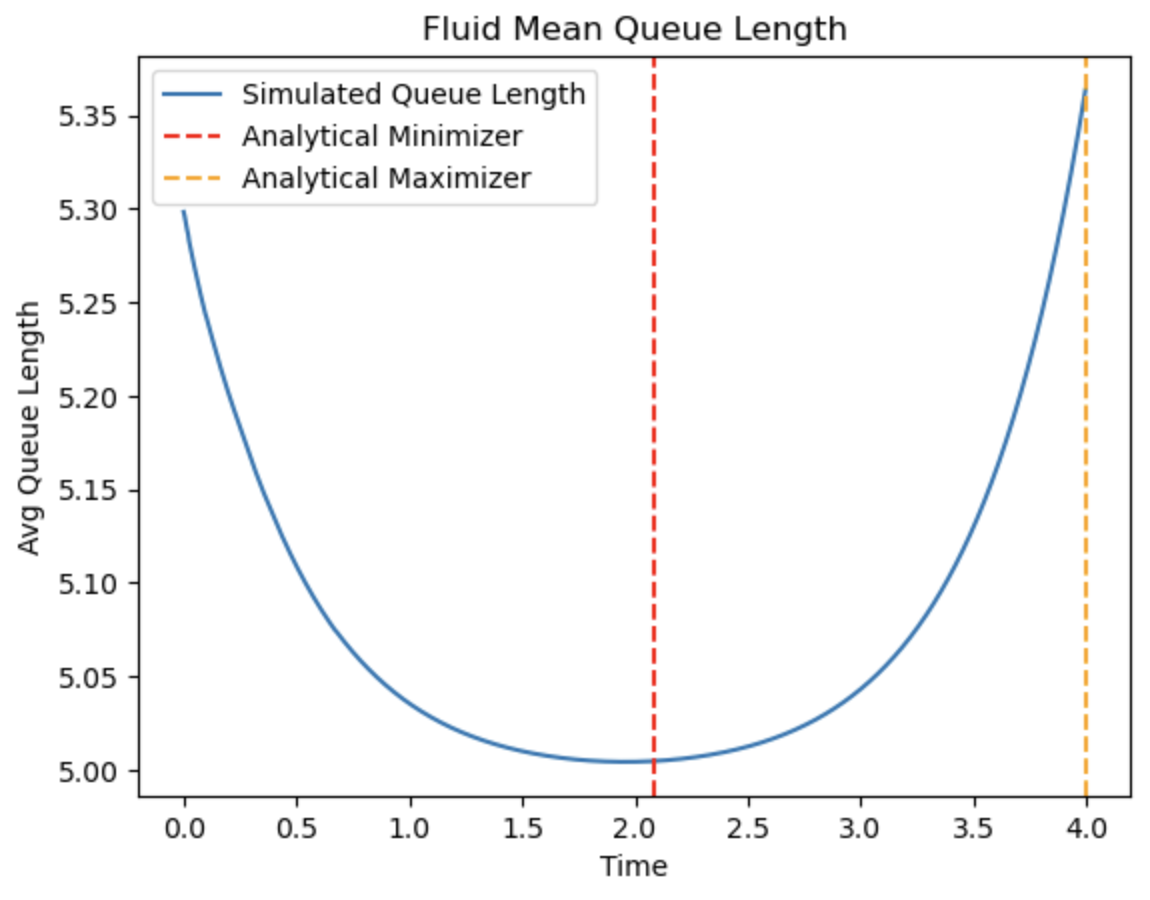}
        \caption{}
        \label{fig:2a}
    \end{subfigure}
    \hfill
    \begin{subfigure}[b]{0.45\textwidth}
        \centering
        \includegraphics[width=\textwidth]{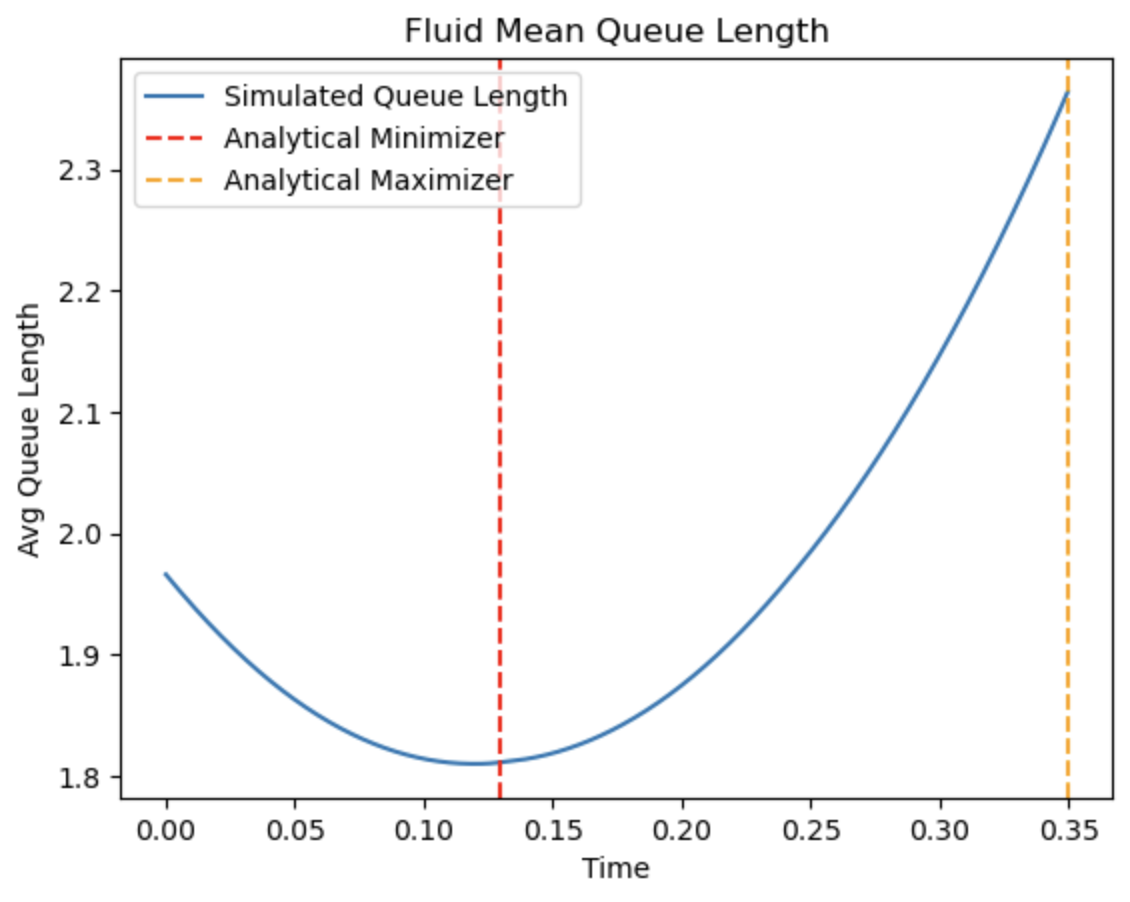}
        \caption{}
        \label{fig:2b}
    \end{subfigure}

    \begin{subfigure}[b]{0.45\textwidth}
        \centering
        \includegraphics[width=\textwidth]{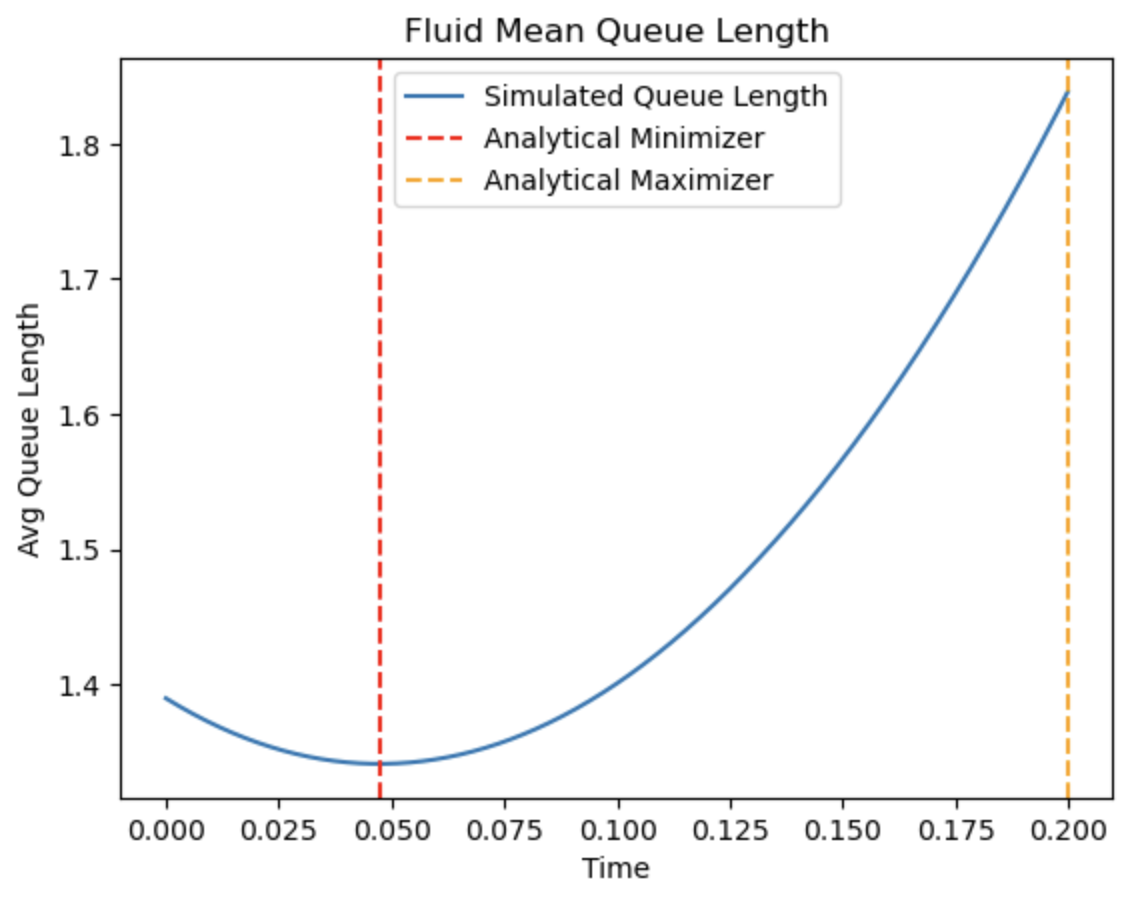}
        \caption{}
        \label{fig:2c}
    \end{subfigure}
    
    \caption{Fluid Mean of Overload System (starts below capacity). \textbf{(a). }$\lambda = 10, \mu = 1, b = 0.5, T = 4, q_0 = 1, \theta = 2, c = 2$. \textbf{(b). }$\lambda = 10, \mu = 1, b = 0.5, T = 0.35, q_0 = 1, \theta = 2, c = 2$. \textbf{(c). }$\lambda = 10, \mu = 1, b = 0.5, T = 0.2, q_0 = 1, \theta = 2, c = 2$.} 
    \label{fig:main_figure}
\end{figure}

\subsection{$q_0 > c, \lambda \leq \mu c$}

Theoretically are two sub-cases corresponding to different time horizons, but to be more comprehensive, we ran experiment on three sub-cases and plotted three plots correspondingly. In reality, notice that when $q_0 \leq c/b$, $\frac{\ln\left( \frac{q_0-\xi_1}{c/b-\xi_1}\right)}{\theta} \leq 0$, and thus the two scenarios reduce to one. So we designate one experiment to that specific case. The figures' correspondence are specified as follows: Figure \ref{fig:3a} corresponds to the case where $q_0 >  c/b$ and $T > \frac{\ln\left( \frac{q_0-\xi_1}{c/b-\xi_1}\right)}{\theta}$; Figure \ref{fig:3b} corresponds to the case where $q_0 >c/b$ and $T \leq \frac{\ln\left( \frac{q_0-\xi_1}{c/b-\xi_1}\right)}{\theta}$; Finally Figure \ref{fig:3c} corresponds to the case where $q_0 \leq c/b$ with a random $T$. Under the parameters used for this case, which are specified in the plots, $\frac{\ln \left( \frac{b(q_0-\xi_1)}{c-b\xi_1} \right)}{\theta} \approx 0.168$ in the first two cases.
\clearpage
\begin{figure}[h!]
    \centering
    \begin{subfigure}[b]{0.45\textwidth}
        \centering
        \includegraphics[width=\textwidth]{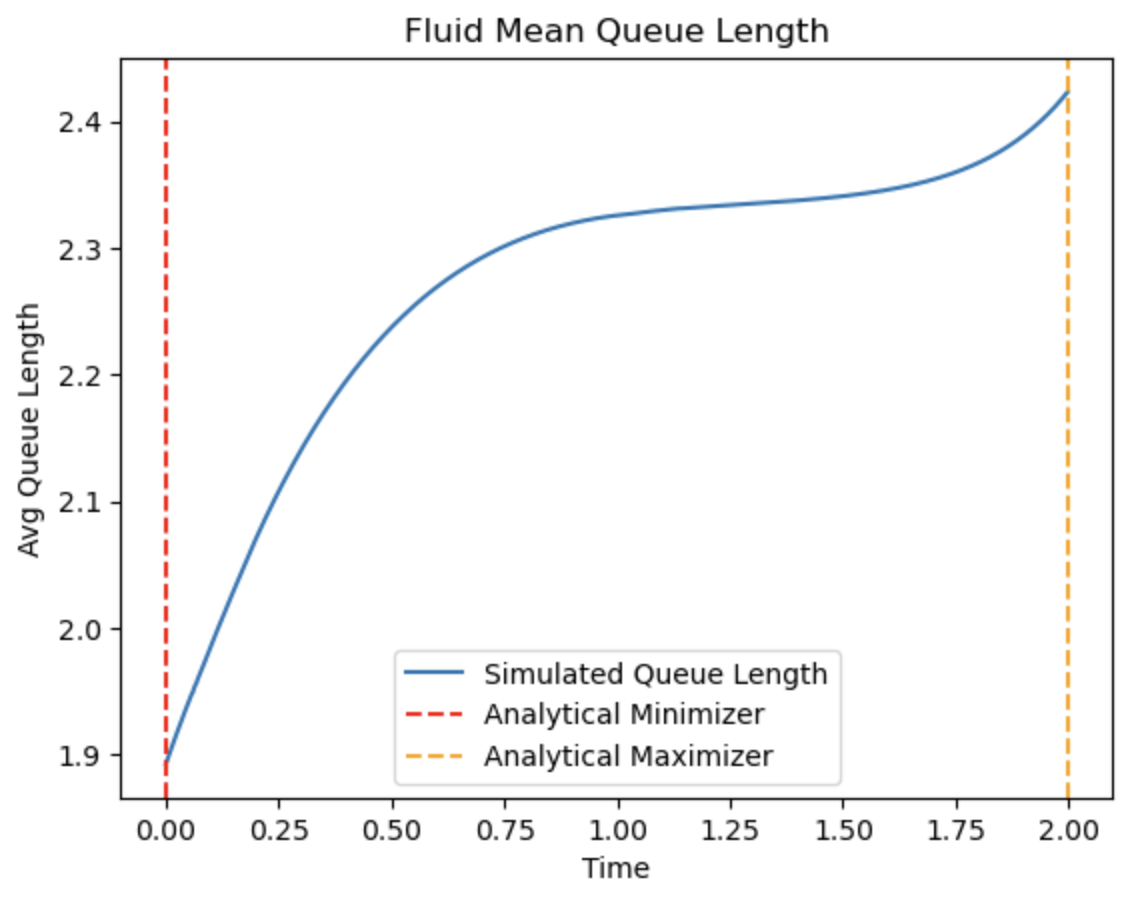}
        \caption{}
        \label{fig:3a}
    \end{subfigure}
    \hfill
    \begin{subfigure}[b]{0.45\textwidth}
        \centering
        \includegraphics[width=\textwidth]{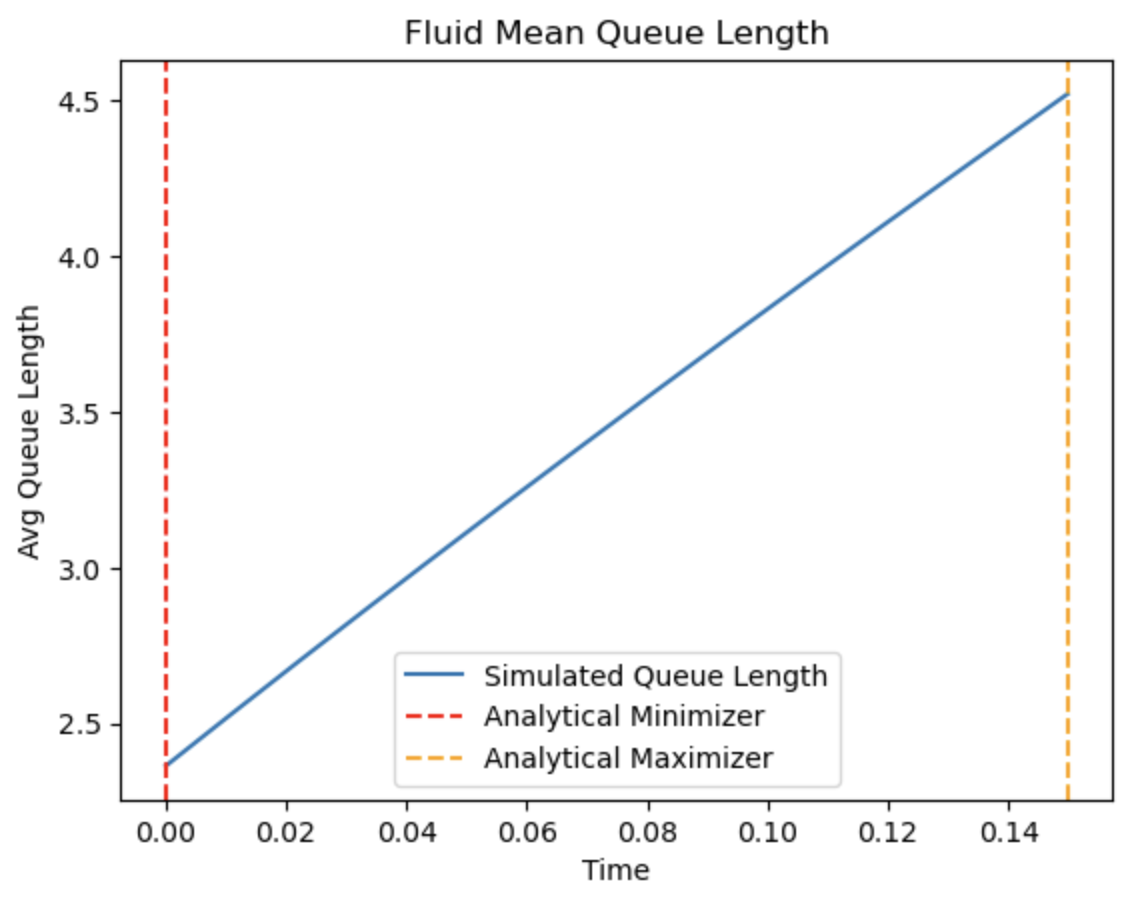}
        \caption{}
        \label{fig:3b}
    \end{subfigure}

    \begin{subfigure}[b]{0.45\textwidth}
        \centering
        \includegraphics[width=\textwidth]{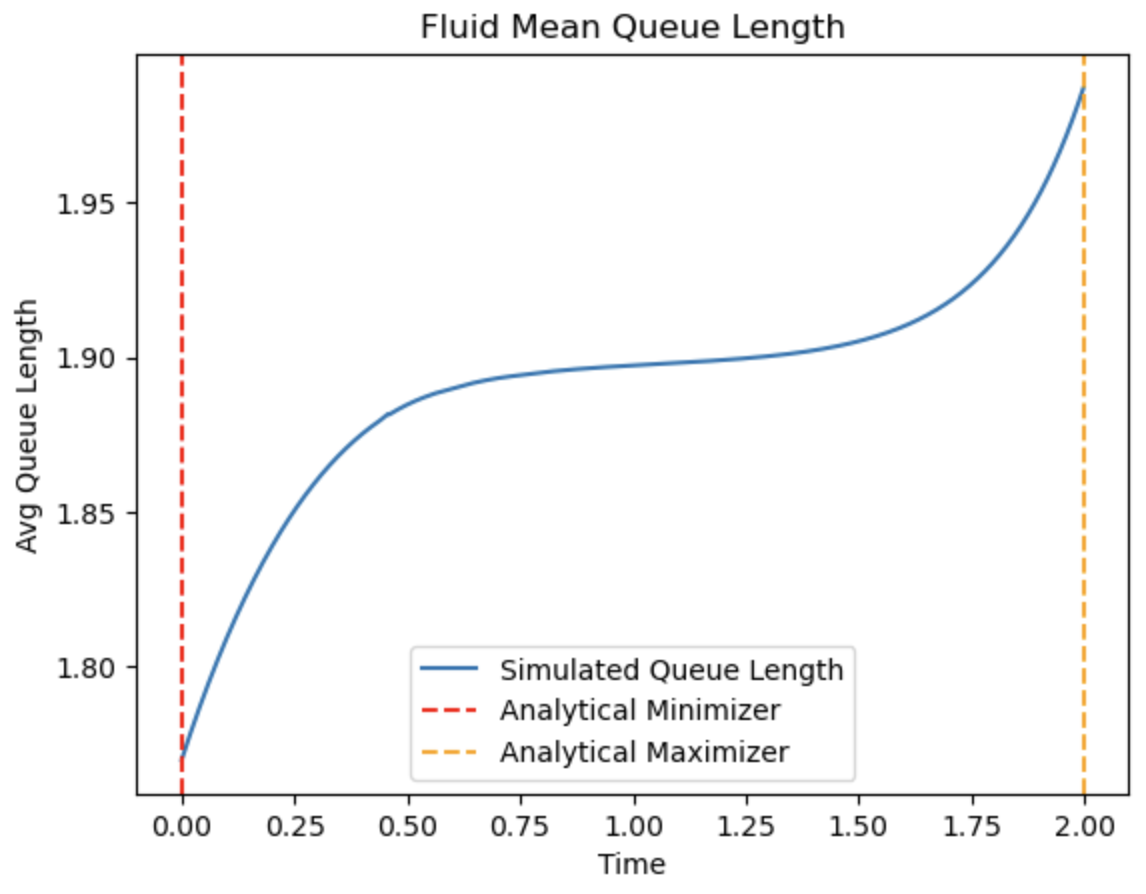}
        \caption{}
        \label{fig:3c}
    \end{subfigure}
    
    \caption{Fluid Mean of Underload System (starts above capacity).  \textbf{(a). }$\lambda = 9, \mu = 5, b = 0.5, T = 2, q_0 = 5, \theta = 2, c = 2$. \textbf{(b). }$\lambda = 9, \mu = 5, b = 0.5, T = 0.15, q_0 = 5, \theta = 2, c = 2$. \textbf{(c). }$\lambda = 9, \mu = 5, b = 0.5, T = 2, q_0 = 3, \theta = 2, c = 2$.}
    \label{fig:main_figure}
\end{figure}

\subsection{$q_0 \leq c, \lambda \leq \mu c$}
Since there is only one interval for the entire case, we only need one numerical experiment. Yet, for the sake of comprehensiveness, we still present two experiments, each with a different time horizon.
\clearpage
\begin{figure}[h!]
    \centering
    \begin{subfigure}[b]{0.45\textwidth}
        \centering
        \includegraphics[width=\textwidth]{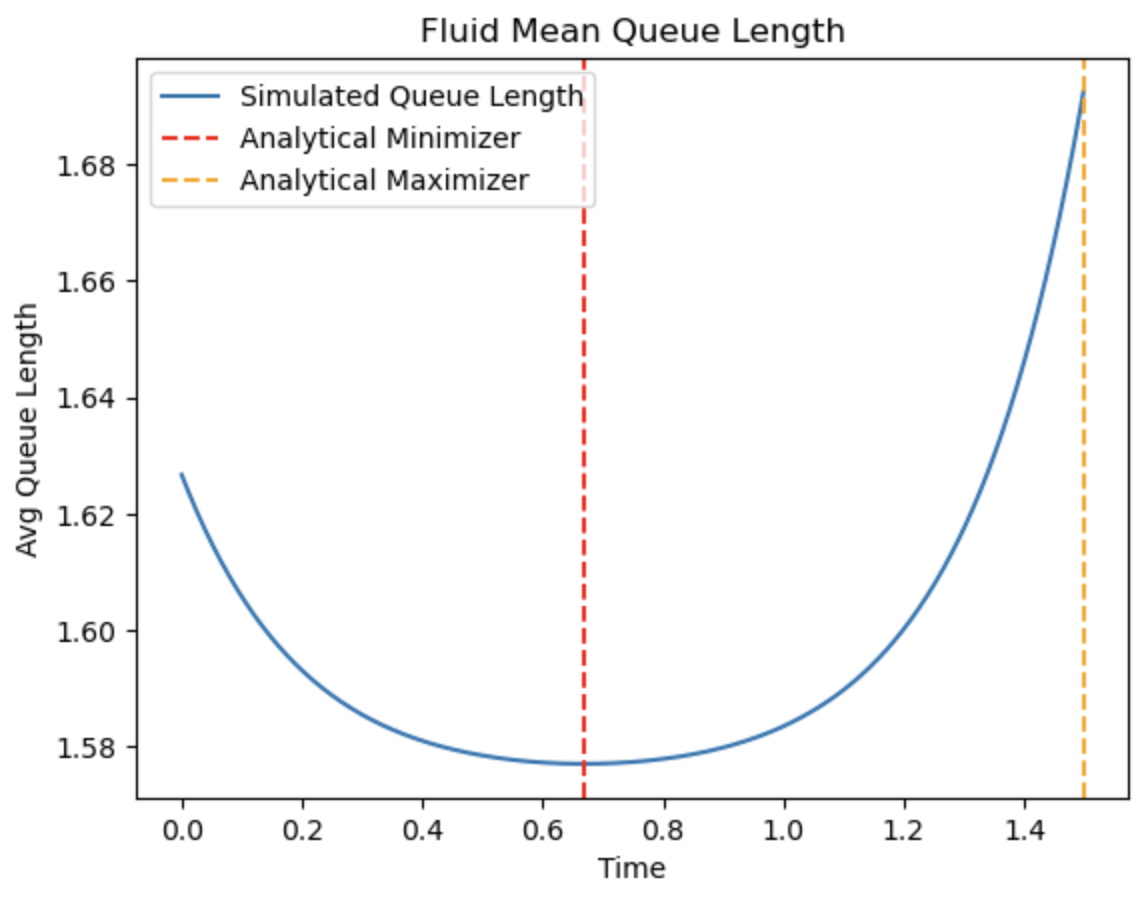}
        \caption{}
        \label{fig:4a}
    \end{subfigure}
    \hfill
    \begin{subfigure}[b]{0.45\textwidth}
        \centering
        \includegraphics[width=\textwidth]{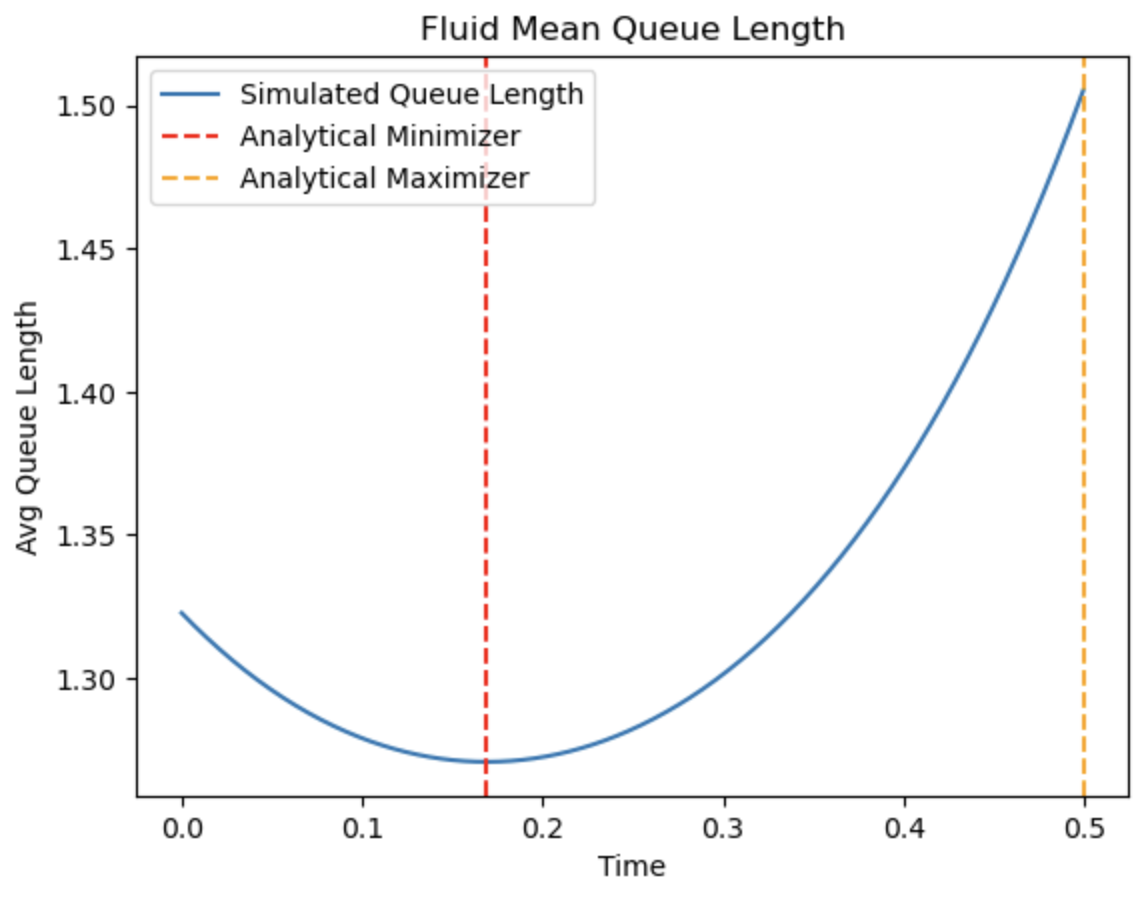}
        \caption{}
        \label{fig:4b}
    \end{subfigure}
    \caption{Fluid Mean of Underload System (starts below capacity).  \textbf{(a). }$\lambda = 9, \mu = 5, b = 0.5, T = 1.5, q_0 = 1, \theta = 2, c = 2$. \textbf{(b). }$\lambda = 9, \mu = 5, b = 0.5, T = 0.5, q_0 = 1, \theta = 2, c = 2$.}
    \label{fig:main_figure}
\end{figure}

\section{Conclusion}\label{secConc}

This paper presents an analysis of impulsive differential equation models applied to queueing systems, with a particular focus on the Erlang-A model. We dive deep and analyzed the system behavior changes corresponding to impulsive control modeled as sudden state changes. Our key contributions include derivation of closed-form expressions for steady-state amplitudes, average queue lengths, and optimal impulse timing (single impulse scenario) under various regimes.

One of the central insights of our work is the identification and quantification of oscillatory behavior in queue lengths, analogous to the charge-discharge cycles of a capacitor. We demonstrate that under certain stability conditions, these oscillations converge to steady patterns, and we quantify the amplitude and average behavior of such dynamics. Furthermore, we solve the inverse optimization problem of determining the optimal time to apply an impulse that either minimizes or maximizes the average queue length over a fixed interval. In the linear case, we provide explicit solutions for the optimal impulse time. For the more complex Erlang-A setting, we characterize several regimes analytically and identify analytical solutions for some of the regimes while proving that numerical methods are necessary for the rest.

While this paper lays the groundwork for understanding oscillatory dynamics and impulse design in impulsive queueing systems, several promising avenues remain for future research. A natural direction to pursue is the multi-impulse design, where multiple impulses instead of one are applied over the time frame $[0,T]$. While the single-impulse case provides valuable insights and tractable closed-form results in certain regimes, generalizing to multiple impulses introduces several challenges. First, the analytical complexity increases significantly. Due to the nature of Erlang-A queues, each impulse introduces multiple scenarios according to the queue length after impulse. Closed-form solutions may no longer be feasible except in highly simplified or specific settings, necessitating robust numerical algorithms or approximation techniques. Moreover, because the effect of each impulse is multiplicative, the Hessian of the cost function is likely going to have mixed signs, making the problem of optimizing multiple impulse times likely non-convex and possible to have multiple local optima.

Another aspect of interest is more non-linear drift model where the dynamics of queue evolution is not linear. For example the following equation is an interesting one to analyze,
$$\shortdot{q}(t) = \lambda - \mu q(t)^{\alpha}, \alpha \neq 1.  $$
The nonlinearity parameter $\alpha$ fundamentally alters how the queue length decays between impulses. For example, for $\alpha > 1$, the departure rate increases superlinearly with the queue length, allowing the system to recover more rapidly from large deviations. We are interested in learning how $\alpha$ affects the optimal time of impulse application. Exploring these directions would deepen our understanding of random impulse, and more specifically how nonlinear system characteristics and multiple interventions influence queue dynamics. And they stand as promising lines of inquiry for future theoretical and practical work.

\section*{Acknowledgements}
Jamol Pender is grateful to the organizers of the Joint Math Meetings Session  (Women in Math)  for introducing him to the topic of impulsive differential equations.  

\pagebreak

\appendix
\section{Proof of Theorem \ref{thm:erlang}}
In all the subsections, $\tau_L^*$ will be used to denote the minimizer on the corresponding intervals of each subsection.
\subsection{\text{$q_0 > c, \lambda > \mu c$}} \label{apdx1.1}
We extend our notation from the linear case and let:
\begin{enumerate}
    \item $\tau$ denote the time where the impulse takes place;
    \item $\xi_1$ be the stationary mean of the system described by Equation \ref{fldmean} given $q_0 > c$. i.e. $$\xi_1 = \frac{\lambda - \mu c + \theta c}{\theta}.$$
    \item $\xi_2$ be the stationary mean of the system described by Equation \ref{fldmean} given $q_0 \leq c$. i.e. $$\xi_2 = \frac{\lambda}{\mu}.$$
\end{enumerate}

We first realize that for different values of $\tau$, the analytical solution of $q_2(t)$ differs, according to whether $b \cdot q_1(\tau) > c$ or $b \cdot q_1(\tau) \leq c$. The intervals that correspond to the above two scenarios are calculated as follows:

For $b \cdot q_1(\tau) > c$,
\begin{align*}
    (q_0 - \xi_1) e^{-\theta t_1} > \frac{c}{b} - \xi_1 \\
    \ln{(q_0 - \xi_1)} - \theta t_1 > \ln{\left(\frac{c}{b} - \xi_1\right)} \\
    t_1 < \ln \left( \frac{\ln \left( \frac{q_0 - \xi_1}{c/b - \xi_1} \right)}{\theta} \right) \quad \quad \text{if } q_0 > \xi_1 \\
    t_1 > \ln \left( \frac{\ln \left( \frac{q_0 - \xi_1}{c/b - \xi_1} \right)}{\theta} \right) \quad \quad \text{if } q_0 < \xi_1.
\end{align*}
Hence,
$$\begin{cases}
    b \cdot q_1(\tau) > c &\text{when } t \in \left[ 0,\ln \left( \frac{\ln \left( \frac{q_0 - \xi_1}{c/b - \xi_1} \right)}{\theta} \right) \right] \text{ and }q_0 > \xi_1\\
    &\text{or when } t \in \left[ \ln \left( \frac{\ln \left( \frac{q_0 - \xi_1}{c/b - \xi_1} \right)}{\theta} \right),T \right] \text{ and }q_0 \leq \xi_1\\
    b \cdot q_1(\tau) \leq c &\text{when } t \in \left[ 0,\ln \left( \frac{\ln \left( \frac{q_0 - \xi_1}{c/b - \xi_1} \right)}{\theta} \right) \right] \text{ and }q_0 \leq \xi_1\\
    &\text{or when } t \in \left[ \ln \left( \frac{\ln \left( \frac{q_0 - \xi_1}{c/b - \xi_1} \right)}{\theta} \right),T \right] \text{ and }q_0 > \xi_1
\end{cases}$$

We investigate each case separately:
\subsubsection{$b \cdot q_1(\tau) > c$}
Now we write down the solution for our differential equation in two pieces:
\begin{align*}
    q_1(t) &= \xi_1 + (q_0 - \xi_1) e^{-\theta t}  &\text{ for } 0 \leq t \leq \tau \\
    q_2(t) &= \xi_1 + \bigg( b\xi_1 + b(q_0-\xi_1)e^{-\theta \tau} - \xi_1 \bigg) e^{-\theta(t-\tau)} \\
    &= \xi_1 + (b-1) \xi_1 \cdot e^{-\theta(t-t_1)} + b(q_0-\xi_1)e^{-\theta t}  &\text{ for } \tau < t \leq T
\end{align*}

By Leibniz rule,
\begin{align*} 
    \frac{d}{d\tau} \int_0^{\tau} q_1(t) dt &= q_1(\tau)\\
    &= \xi_1 + (q_0 - \xi_1) e^{-\theta \tau} \\
    \frac{d}{d\tau} \int_{\tau}^T q_2(t, \tau) dt &= - \frac{d}{d\tau} \int_T^{\tau} q_2(t, \tau) dt \\
    &= - q_2(\tau, \tau) - \int_T^{\tau} \frac{\partial}{\partial \tau} q_2(t, \tau) dt\\
    &= -b \xi_1 - b(q_0-\xi_1)e^{-\theta \tau} - \int_T^{\tau}  (b-1) \xi_1 \theta e^{-\theta (t-\tau)} dt\\
    &= -b \xi_1 - b(q_0-\xi_1)e^{-\theta \tau} + (b-1) \xi_1 - (b-1) \xi_1  e^{-\theta (T-\tau)}\\
    &= -\xi_1 - b(q_0-\xi_1)e^{-\theta \tau} - (b-1) \xi_1  e^{-\theta (T-\tau)}
\end{align*}

Now we discuss two cases:

(i) When $q_0 > \xi_1$, we can see that the derivative of the total queue length, being $$\frac{d}{d\tau} \int_0^{\tau} q_1(t) dt +\frac{d}{d\tau} \int_{\tau}^T q_2(t, \tau) dt = (1-b)(q_0-\xi_1)e^{-\theta \tau} - (b-1) \xi_1  e^{-\theta (T-\tau)},$$
is strictly positive. Recall that when $q_0 > \xi_1$, $\tau \in \left[0, \ln \left( \frac{\ln \left( \frac{q_0 - \xi_1}{c/b - \xi_1} \right)}{\theta} \right) \right]$, we have:
$$\underset{\tau}{\mathrm{argmin}}_{\tau} \int_0^t q(t) dt = 0,$$
$$\underset{\tau}{\mathrm{argmax}}_{\tau} \int_0^t q(t) dt = \min \left( T, \ln \left( \frac{\ln \left( \frac{q_0 - \xi_1}{c/b - \xi_1} \right)}{\theta} \right) \right).$$

(ii) When $q_0 \leq \xi_1$, we can see that the derivative of the total queue length, being $$\frac{d}{d\tau} \int_0^{\tau} q_1(t) dt +\frac{d}{d\tau} \int_{\tau}^T q_2(t, \tau) dt = (1-b)(q_0-\xi_1)e^{-\theta \tau} - (b-1) \xi_1  e^{-\theta (T-\tau)},$$
has a solution when set to 0.
\begin{align*}
    (1-b)(q_0-\xi_1)e^{-\theta \tau} = (b-1) \xi_1  e^{-\theta (T-\tau)} \\
    \ln(\xi_1 - q_0) - \theta \tau = \ln(\xi_1) - \theta T + \theta \tau \\
    \tau^{OPT} = \frac{\ln(\xi_1-q_0)-\ln(\xi_1)}{2 \theta}+ \frac{T}{2}
\end{align*}
Once again recall that  when $q_0 < \xi_1, \tau \in \left[ \ln \left( \frac{\ln \left( \frac{q_0 - \xi_1}{c/b - \xi_1} \right)}{\theta} \right),T\right]$, 
$$\underset{\tau}{\mathrm{argmin}}_{\tau} \int_0^t q(t) dt =  \max\left( \tau^{OPT},  \ln \left( \frac{\ln \left( \frac{q_0 - \xi_1}{c/b - \xi_1} \right)}{\theta} \right)\right),$$
$$\underset{\tau}{\mathrm{argmax}}_{\tau} \int_0^t q(t) dt = T$$

\subsubsection{$b \cdot q_1(\tau) < c$} \label{apdx: poly}
We now analyze the situation where the jump occurs on the other half of the interval $[0,T]$. Since $b \cdot q_1(\tau) < c$, we should discuss first, the case where $t_{2,q_1(\tau)}^* > T-\tau$ and then, the case where $t_{2,q_1(\tau)}^* < T-\tau$. 

(i) When $\tau > T - t_{2,q_1(\tau)}^*$,
\begin{align*}
    q_1(t) &= \xi_1 + (q_0 - \xi_1) e^{-\theta t}  &\text{ for } 0 \leq t \leq \tau \\
    q_2(t) &= \xi_2 + \bigg( b\xi_1 + b(q_0-\xi_1)e^{-\theta \tau} - \xi_2 \bigg) e^{-\mu(t-\tau)} &\text{ for } \tau < t \leq T
\end{align*}

By the Leibniz rule,

\begin{align*} 
    \frac{d}{d\tau} \int_0^{\tau} q_1(t) dt &= q_1(\tau) = \xi_1 + (q_0 - \xi_1) e^{-\theta \tau} \\
    \frac{d}{d\tau} \int_{\tau}^T q_2(t, \tau) dt &= - \frac{d}{d\tau} \int_T^{\tau} q_2(t, \tau) dt \\
    &= - q_2(\tau, \tau) - \int_T^{\tau} \frac{\partial}{\partial \tau} q_2(t, \tau) dt\\
    &= - \xi_2 - b\xi_1 + \xi_2 - b(q_0 - \xi_1) e^{-\theta \tau}\\
    &- \int_T^{\tau} (b \xi_1 - \xi_2) \mu e^{-\mu (t-\tau)} + b(q_0 - \xi_1) e^{-\mu t} (\mu - \theta) e^{(\mu-\theta) \tau} dt\\
    &= -b\xi_1 - b(q_0 - \xi_1) e^{-\theta \tau} + (b\xi_1 - \xi_2) - (b\xi_1 - \xi_2)e^{-\mu(T-\tau)}\\
    &+ \frac{\mu-\theta}{\mu} b(q_0 - \xi_1) e^{-\theta \tau} -\frac{\mu-\theta}{\mu} b(q_0 - \xi_1) e^{(\mu - \theta)\tau - \mu T} \\
    &= -\xi_2-\frac{\theta}{\mu} b(q_0-\xi_1)e^{-\theta \tau}- (b\xi_1 - \xi_2)e^{-\mu(T-\tau)}-\frac{\mu-\theta}{\mu} b(q_0 - \xi_1) e^{(\mu - \theta)\tau - \mu T}
\end{align*}
When one substitute: $e^{-\theta \tau}$ with $x$; $-\frac{\mu}{\theta}$ with $m$; $\xi_1-\xi_2$ with $b$; $(1-\frac{\theta}{\mu}) b(q_0-\xi_1)$ with $b$; $-(b\xi_1 - \xi_2)e^{-\mu T}$ with $c'$; and $-\frac{\mu-\theta}{\mu} b(q_0 - \xi_1) e^{- \mu T}$ with $d$, one can realize that the expression for $\frac{d}{d\tau} \int_0^{\tau} q_1(t) dt + \frac{d}{d\tau} \int_{\tau}^T q_2(t, \tau) dt$ is actually the polynomial $a+bx-cx^m-dx^{m+1}$. By Abel-Ruffini theorem, there is no solution in radicals to general polynomial equations of degree five or higher with arbitrary coefficients, but the existence of a solution is guaranteed. By simple convexity analysis, the solution is the minimizer $\tau_L^*$ while $\tau_U^* = T$ results in an average queue length of $\int_0^T \xi_1 + (q_0-\xi_1)e^{-\theta t} dt$ that dominates all cases with impulses in $[0,T)$. Intuitively, this is equivalent to no impulse within the interval and hence the queue length never drops and remains at its maximum.

(ii) When $\tau \leq T - t_{2,q_1(\tau)}^*$,
In this case, the time horizon allows the system to reach capacity after impulse. The solution to Equation \ref{abeqn1} become:
\begin{align*}
    q_1(t) &= \xi_1 + (q_0 - \xi_1) e^{-\theta t}  &\text{ for } 0 \leq t \leq \tau \\
    q_2(t) &= \xi_2 + \bigg( bq_1(\tau) - \xi_2 \bigg) e^{-\mu(t-\tau)} &\text{ for } \tau < t \leq \tau+t_{2,bq_1(\tau)}^*\\
    q_3(t) &= \xi_1 + (c-\xi_1) e^{-\theta(t-\tau-t_{2,bq_1(\tau)}^*)} &\text{ for } \tau+t_{2,bq_1(\tau)}^* < t \leq T
\end{align*}

The total queue length corresponding to the three parts are:
\begin{align}
    \int_0^\tau q_1(t) dt &= \xi_1 \tau + \frac{q_0-\xi_1}{\theta}(1-e^{-\theta\tau}) \nonumber\\
    \int_\tau^{\tau+t_{2,bq_1(\tau)}^*} q_2(t) dt &= \xi_2 \cdot t_{2,bq_1(\tau)}^* + \frac{bq_1(\tau)-\xi_2}{\mu} (1-e^{-\mu \cdot t_{2,bq_1(\tau)}^*}) \nonumber \\
    &= \xi_2 \cdot t_{2,bq_1(\tau)}^* - \frac{c-\xi_2}{\mu} + \frac{b\xi_1 + b(q_0-\xi_1)e^{-\theta \tau} - \xi_2}{\mu} \label{eqn: reduction}\\
    \int_{\tau+t_{2,bq_1(\tau)}^*}^T q_3(t) dt &= \xi_1(T- t_{2,bq_1(\tau)}^*-\tau) + \frac{c-\xi_1}{\theta}(1-e^{-\theta(T- t_{2,bq_1(\tau)}^*-\tau)}) \nonumber
\end{align}
Here, the reason behind taking the extra step of computing the integral without using the Leibniz rule directly lies in Equation \ref{eqn: reduction}. Due to the definition of $t_{2,bq_1(\tau)}^*$, we are able to simplify the expression which would be much harder if we directly apply Leibniz rule. Hence,
\begin{align*}
    \frac{d}{d\tau} \int_0^\tau q_1(t) dt &= \xi_1 + (q_0 - \xi_1) e^{-\theta \tau} \\
    \frac{d}{d\tau} \int_\tau^{\tau+t_{2,bq_1(\tau)}^*} q_2(t) dt &= \xi_2 \cdot (t_{2,bq_1(\tau)}^*)' - \frac{\theta}{\mu}b(q_0-\xi_1)e^{-\theta\tau}\\
    \frac{d}{d\tau}\int_{\tau+t_{2,bq_1(\tau)}^*}^T q_3(t) dt &= -\xi_1 (t_{2,bq_1(\tau)}^*)'-\xi_1 - (\xi_1-c) e^{-\theta(T- t_{2,bq_1(\tau)}^*-\tau)} (1+(t_{2,bq_1(\tau)}^*)')\\
    \frac{d}{d\tau} \int_0^T q(t) dt &= \frac{d}{d\tau} T \cdot J(\tau) \\
    &=\frac{d}{d\tau} \int_0^\tau q_1(t) dt + \frac{d}{d\tau} \int_\tau^{\tau+t_{2,bq_1(\tau)}^*} q_2(t) dt + \frac{d}{d\tau}\int_{\tau+t_{2,bq_1(\tau)}^*}^T q_3(t) dt\\
    &= (\xi_2-\xi_1)(t_{2,bq_1(\tau)}^*)' + \left( 1- \frac{b \theta}{\mu} \right) (q_0-\xi_1)e^{-\theta \tau} \\
    &- (c-\xi_1) (1+(t_{2,bq_1(\tau)}^*)')e^{-\theta(T- t_{2,bq_1(\tau)}^*-\tau)}
\end{align*}
where
\begin{align}
    (t_{2,bq_1(\tau)}^*)' = \frac{\theta}{\mu} \frac{-b(q_0-\xi_1)}{bq_1(\tau)-\xi_2} e^{-\theta \tau} \label{eqn: num}
\end{align}
Due to the interactive nature between $\lambda$, $\mu$, $\theta$, and $c$ in this section, we are not able to determine the monotonicity of the general $\frac{d}{d\tau} \int_0^T q(t) dt$ and numerical solution is hence needed to solve Equation \ref{eqn: num}. (Although an analytical solution is theoretically attainable by partitioning the problem into an extensive number of sub-cases, the sheer number and complexity of these sub-cases render such an approach unrealistic. Therefore, a numerical solution is preferred for both tractability and efficiency.)
\subsection{\text{$q_0 \leq c, \lambda > \mu c$}}
We first analyze the interval division as in the previous section:
\begin{align*}
    \xi_2 + (q_0-\xi_2)e^{-\mu \tau} < \frac{c}{b}\\
    e^{-\mu \tau}  > \frac{c/b-\xi_2}{q_0-\xi_2}\\
    -\mu\tau > \ln\left( \frac{c/b-\xi_2}{q_0-\xi_2}\right)\\
    \tau < \frac{\ln\left(\frac{q_0-\xi_2}{c/b-\xi_2}\right)}{\mu}
\end{align*}
Note that there are no need to divide into sub-cases as in the previous section because of the assumption of this section $\xi_2 = \frac{\lambda}{\mu} > c > q_0$. Here, the interval $\left[0,t_{2,q_0}^*\right]$ corresponds to the case where the system is below capacity, i.e. $q_1(t) \leq c$ and hence the system after impulse can only be below capacity ($bq_1(\tau) \leq c$). We then first consider the scenario where the time horizon is long enough. In this case, $T$ supports recovery to the capacity $c$ after impulse. Then, because of the nature of the Erlang-A system, the interval $\left[\frac{\ln\left(\frac{q_0-\xi_2}{c/b-\xi_2}\right)}{\mu},T\right]$ is further divided into $\left[\frac{\ln\left(\frac{q_0-\xi_2}{c/b-\xi_2}\right)}{\mu}, t_{2,q_0}^*\right]$, and $\left[t_{2,q_0}^*, T\right]$. Then, in the scenario where $T \in \left[t_{2,q_0}^*,\frac{\ln \left( \frac{b(q_0-\xi_1)}{c-b\xi_1} \right)}{\theta} \right]$, we have the intervals $[0, t_{2,q_0}^*]$ and $ [t_{2,q_0}^*,T]$. Finally for $T\in[0,t_{2,q_0}^*]$, We only have one interval $[0,T].$

\subsubsection{$\tau \in I_{2,1} = \left[0,t_{2,q_0}^*\right]$} \label{apdx: 2.1}
This is the common case for $T > t_{2,q_0}^*$. Assembling the analysis in the previous sections, we first write down the solution to Equation \ref{abeqn1}:
\begin{align*}
    q_1(t) &= \xi_2 + (q_0 - \xi_2) e^{-\mu t}  &\text{ for } 0 \leq t \leq \tau \\
    q_2(t) &= \xi_2 + \bigg( bq_1(\tau) - \xi_2 \bigg) e^{-\mu(t-\tau)} &\text{ for } \tau < t \leq \tau+t_{2,bq_1(\tau)}^*\\
    q_3(t) &= \xi_1 + (c-\xi_1) e^{-\theta(t-\tau-t_{2,bq_1(\tau)}^*)} &\text{ for } \tau+t_{2,bq_1(\tau)}^* < t \leq T
\end{align*}
The total queue length corresponding to the three parts are:
\begin{align}
    \int_0^\tau q_1(t) dt &= \xi_2 \tau + \frac{q_0-\xi_2}{\mu}(1-e^{-\mu\tau}) \nonumber\\
    \int_\tau^{\tau+t_{2,bq_1(\tau)}^*} q_2(t) dt &= \xi_2 \cdot t_{2,q_0}^* + \frac{bq_1(\tau)-\xi_2}{\mu} (1-e^{-\mu \cdot t_{2,bq_1(\tau)}^*}) \nonumber \\
    &= \xi_2 \cdot t_{2,bq_1(\tau)}^* - \frac{c-\xi_2}{\mu} + \frac{b\xi_2 + b(q_0-\xi_2)e^{-\mu \tau} - \xi_2}{\mu} \\
    \int_{\tau+t_{2,bq_1(\tau)}^**}^T q_3(t) dt &= \xi_1(T- t_{2,bq_1(\tau)}^*-\tau) + \frac{c-\xi_1}{\theta}(1-e^{-\theta(T- t_{2,bq_1(\tau)}^*-\tau)})
\end{align}

Hence, taking derivative gives:
\begin{align*}
    \frac{d}{d\tau} \int_0^\tau q_1(t) dt &= \xi_2 + (q_0 - \xi_2) e^{-\mu \tau} \\
    \frac{d}{d\tau} \int_\tau^{\tau+t_{2,bq_1(\tau)}^*} q_2(t) dt &= \xi_2 \cdot (t_{2,bq_1(\tau)}^*)' - b(q_0-\xi_2)e^{-\mu\tau}\\
    \frac{d}{d\tau}\int_{\tau+t_{2,bq_1(\tau)}^*}^T q_3(t) dt &= -\xi_1 (1+(t_{2,bq_1(\tau)}^*)') - (c-\xi_1) e^{-\theta(T- t_{2,bq_1(\tau)}^*-\tau)} (1+(t_{2,bq_1(\tau)}^*)')\\
    \frac{d}{d\tau} \int_0^T q(t) dt &= \frac{d}{d\tau} T \cdot J(\tau) \\
    &=\frac{d}{d\tau} \int_0^\tau q_1(t) dt + \frac{d}{d\tau} \int_\tau^{\tau+t_{2,bq_1(\tau)}^*} q_2(t) dt + \frac{d}{d\tau}\int_{\tau+t_{2,bq_1(\tau)}^*}^T q_3(t) dt\\
    &= (\xi_2-\xi_1)(1+(t_{2,bq_1(\tau)}^*)') + \left( 1- b \right) (q_0-\xi_1)e^{-\mu \tau} - (c-\xi_1) (1+(t_{2,bq_1(\tau)}^*)')e^{-\theta(T- t_{2,bq_1(\tau)}^*-\tau)}
\end{align*}
where
\begin{align}
    (t_{2,q_0}^*)' = \frac{-b(q_0-\xi_1)}{bq_1(\tau)-\xi_2} e^{-\mu \tau} 
\end{align}
Furthermore, since $b<1, bq_1(\tau) <c< \xi_2$, we notice that
\begin{align*}
    1+(t_{2,q_0}^*)' = \frac{(b-1)\xi_2}{bq_1(\tau)-\xi_2} > 0.
\end{align*}
And since $\lambda>\mu c$, it follows that $c-\xi_1 < 0$, and consequently $\left( 1- b \right) (q_0-\xi_1) < 0$, $-(c-\xi_1) (1+(t_{2,q_0}^*)') > 0$. This implies that the derivative $$\frac{d}{d\tau} \int_0^T q(t) dt$$ increases in value as $\tau$ increases. It is easy to verify that the derivative at $\tau = \frac{\ln\left( \frac{c/b-\xi_2}{q_0-\xi_2}\right)}{\mu}$ is negative, which suggests that $\frac{d}{d\tau} \int_0^T q(t) dt < 0$ for all $\tau \in \left[0,t_{2,q_0}^*\right]$. The optimizers $\tau_L^*,\tau_U^*$ on this interval are hence $$\tau_{L}^* =t_{2,q_0}^*, \quad \quad \tau_U^* = 0.$$

\subsubsection{$\tau \in I_{2,2} = \left[t_{2,q_0}^*,\frac{\ln \left( \frac{b(q_0-\xi_1)}{c-b\xi_1} \right)}{\theta}\right]$ with $T > \frac{\ln \left( \frac{b(q_0-\xi_1)}{c-b\xi_1} \right)}{\theta}$}
The solution to Equation \ref{abeqn1} in this case is:
\begin{align*}
    q_1(t)& = \xi_2 + (q_0 - \xi_2) e^{-\mu t}&\text{ for } 0 \leq t < t_{2,q_0}^*\\
    q_2(t)& = \xi_1 + (c - \xi_1) e^{-\theta (\tau - t_{2,q_0}^*)} &\text{ for } t_{2,q_0}^* \leq t < \tau\\
    q_3(t)& =  \xi_2 + \left( b \xi_1 + b ( c - \xi_1) e^{-\theta (\tau - t_{2,q_0}^*)} - \xi_2 \right) e^{-\mu (t - \tau)} &\text{ for } \tau \leq t < \tau + t_{2,bq_2(\tau)}^{*}\\
    q_4(t) &= \xi_1 + (c - \xi_1) e^{-\theta (t - \tau - t_{2,bq_2(\tau)}^{*})}&\text{ for } \tau + t_{2,bq_2(\tau)}^{*} \leq t \leq T
\end{align*}

Taking integral to compute the four components of $T\cdot J(\tau)$:
\begin{align*}
    \int_0^{t_{2,q_0}^*} q_1(t) dt &= \xi_2 t_{2,q_0}^* + (q_0 - \xi_2)(1 - e^{-\mu t_{2,q_0}^*})\\
    \int_{t_{2,q_0}^*}^\tau q_2(t) dt &= \xi_1 (\tau - t_{2,q_0}^*) + \frac{c - \xi_1}{\theta} \left( 1 - e^{-\theta (\tau - t_{2,q_0}^*)} \right)\\
    \int_\tau^{\tau + t_{2,bq_2(\tau)}^{*}} q_3(t) dt &= \xi_2 \cdot t_{2,bq_2(\tau)}^{*} + \frac{bq_2(\tau) - \xi_2}{\mu} (1 - e^{-\mu \cdot t_{2,bq_2(\tau)}^{*}})\\
    &=  \xi_2 \cdot t_{2,bq_2(\tau)}^{*}  + \frac{bq_2(\tau) - \xi_2}{\mu} - \frac{c - \xi_2}{\mu}\\
    \int_{\tau + t_{2,bq_2(\tau)}^{*}}^T q_4(t) dt &= \xi_1 (T - \tau-t_{2,bq_2(\tau)}^{*}) + \frac{c - \xi_1}{\theta} \left( 1 - e^{-\theta (T - \tau -t_{2,bq_2(\tau)}^{*})} \right)
\end{align*}

Taking derivatives gives:
\begin{align*}
\frac{d }{d \tau} \int_0^{t_{2,q_0}^*} q_1(t) dt&= 0 \\
\frac{d }{d \tau}\int_{t_{2,q_0}^*}^\tau q_2(t) dt &= \xi_1 + (c - \xi_1) e^{-\theta (\tau - t_{2,q_0}^*)} \\
\frac{d }{d \tau}\int_\tau^{\tau + t_{2,bq_2(\tau)}^{*}} q_3(t) dt &= \xi_2 (t_{2,bq_2(\tau)}^{*})' - \frac{\theta}{\mu} b (c - \xi_1) e^{-\theta (\tau - t_{2,q_0}^*)} \\
\frac{d }{d \tau}\int_{\tau + t_{2,bq_2(\tau)}^{*}}^T q_4(t) dt  &= -(c - \xi_1)(t_{2,bq_2(\tau)}^{*})' e^{-\theta (T - \tau - t_{2,bq_2(\tau)}^{*})} - \xi_1 (t_{2,bq_2(\tau)}^{*})'\\
\frac{d}{d\tau} T \cdot J(\tau) &= \xi_1 + (\xi_2-\xi_1)(t_{2,bq_2(\tau)}^{*})' \\
&+ \left( \frac{\theta b}{\mu} - 1 \right) (\xi_1-c) e^{-\theta(\tau-t_{2,bq_2(\tau)}^{*})} + (\xi_1-c)(t_{2,bq_2(\tau)}^{*})'e^{-\theta (T - \tau - t_{2,bq_2(\tau)}^{*})}
\end{align*}
where 
\begin{align*}
    (t_{2,bq_2(\tau)}^{*})' &= \frac{c-\xi_2}{\mu(bq_2(\tau)-\xi_2)} \frac{1}{c-\xi_2} q_0'\\
    &= -\frac{\theta}{\mu} \frac{b(c-\xi_1)}{bq_2(\tau)-\xi_2} e^{-\theta(\tau-t_{2,bq_2(\tau)}^{*})}.
\end{align*}
This is another case demonstrating the difficulty imposed on analytical solution derivation by the piecewise nature of the system's dynamics. More specifically, in this case, just like in Section \ref{apdx: poly}, each segment converges to a different steady state, manifesting in the above expressions as the simultaneous appearance of $\xi_1$ and $\xi_2$, $\mu$ and $\theta$. Moreover, the number of possible combinations of relationship among $\lambda$, $\mu$, $\theta$, and $c$ renders a general closed form analytical solution unrealistic and a numerical solution is the most efficient. As for the analysis of $\tau_U^*$, the reasoning follows from Section \ref{apdx1.1} that $\tau_U^* = T$ results in an average queue length of $\int_0^T \xi_2 + (q_0-\xi_2)e^{-\mu t} dt$ that dominates all cases with impulses in $[0,T)$.

\subsubsection{$\tau \in I_{2,3} =  \left[\frac{\ln \left( \frac{b(q_0-\xi_1)}{c-b\xi_1} \right)}{\theta},T\right]$ with $T > \frac{\ln \left( \frac{b(q_0-\xi_1)}{c-b\xi_1} \right)}{\theta}$}

The solution to Equation \ref{abeqn1} in this case is:
\begin{align*}
    q_1(t)& = \xi_2 + (q_0 - \xi_2) e^{-\mu t}&\text{ for } 0 \leq t < t_{2,q_0}^*\\
    q_2(t)& = \xi_1 + (c - \xi_1) e^{-\theta (\tau - t_{2,q_0}^*)} &\text{ for } t_{2,q_0}^* \leq t < \tau\\
    q_3(t)& =  \xi_1 + \left( b \xi_1 + b ( c - \xi_1) e^{-\theta (\tau - t_{2,q_0}^*)} - \xi_1 \right) e^{-\theta (t - \tau)} &\text{ for } \tau \leq t \leq T
\end{align*}

Taking integral to compute the four components of $T\cdot J(\tau)$:
\begin{align*}
    \int_0^{t_{2,q_0}^*} q_1(t) dt &= \xi_2 t_{2,q_0}^* + \frac{q_0 - \xi_2}{\mu}(1 - e^{-\mu t_{2,q_0}^*})\\
    \int_{t_{2,q_0}^*}^\tau q_2(t) dt &= \xi_1 (\tau - t_{2,q_0}^*) + \frac{c - \xi_1}{\theta} \left( 1 - e^{-\theta (\tau - t_{2,q_0}^*)} \right)\\
    \int_\tau^{\tau + t_{2,bq_2(\tau)}^{*}} q_3(t) dt &= \xi_1(T-\tau)+(b-1)\xi_1 \cdot (1-e^{-\theta(T-\tau)}) + \frac{b(c-\xi_1)}{\theta}(e^{-\theta(\tau-t_{2,bq_2(\tau)}^{*})}-e^{-\theta(T-t_{2,bq_2(\tau)}^{*})}).
\end{align*}

Taking derivatives gives:
\begin{align*}
\frac{d }{d \tau} \int_0^{t_{2,q_0}^*} q_1(t) dt&= 0 \\
\frac{d }{d \tau}\int_{t_{2,q_0}^*}^\tau q_2(t) dt &= \xi_1 + (c - \xi_1) e^{-\theta (\tau - t_{2,q_0}^*)} \\
\frac{d }{d \tau}\int_\tau^T q_3(t) dt &= -\xi_1 - (b-1) \xi_1 e^{-\theta(T-\tau)} - b(c-\xi_1)e^{-\theta(\tau-t_{2,bq_2(\tau)}^{*})}\\
\frac{d}{d\tau} T \cdot J(\tau) &= (1-b)(c-\xi_1) e^{-\theta(\tau-t_{2,bq_2(\tau)}^{*})} + (1-b)\xi_1 e^{-\theta(T-\tau)}
\end{align*}

Notice that $(1-b)(c-\xi_1) < 0$ and $(q-b)\xi_1 > 0$, so $\frac{d}{d\tau} T \cdot J(\tau) = 0$ has a solution:
\begin{align*}
    -(c-\xi_1) e^{-\theta(\tau-t_{2,bq_2(\tau)}^{*})} =\xi_1 e^{-\theta(T-\tau)}\\
    \ln(\xi_1-c) -\theta\tau + \theta t_{2,bq_2(\tau)}^{*} = \ln(\xi_1) -\theta T+\theta\tau\\
    \tau^{OPT} = \frac{T+t_{2,bq_2(\tau)}^{*}}{2} + \frac{\ln\left(\frac{\xi_1-c}{\xi_1}\right)}{2\theta}
\end{align*}
Simple convexity analysis guarantees that $\tau_L^* = \tau^{OPT}$. As for the analysis of $\tau_U^*$, the reasoning follows from Section \ref{apdx1.1} that $\tau_U^* = T$ results in an average queue length of $\int_0^T \xi_2 + (q_0-\xi_2)e^{-\mu t} dt$ that dominates all cases with impulses in $[0,T)$.

\subsubsection{$\tau \in I_{2,4} =  [t_{2,q_0}^*,T]$ with $T \in \left[t_{2,q_0}^*,\frac{\ln \left( \frac{b(q_0-\xi_1)}{c-b\xi_1} \right)}{\theta}\right]$}
This section corresponds to the case where the system reaches capacity before impulse but isn't capable of reviving back to capacity due to time horizon. The solution to Equation \ref{abeqn1} becomes:
\begin{align*}
    q_1(t) &= \xi_2 + (q_0 - \xi_2) e^{-\mu t}  &\text{ for } 0 \leq t \leq t_{2,q_0}^* \\
    q_2(t) &= \xi_1 + (c-\xi_1) e^{-\theta(t-t_{2,q_0}^*)} &\text{ for } t_{2,q_0}^* < t \leq \tau\\
    q_3(t) &= \xi_2 + (bq_2(\tau)-\xi_2) e^{-\theta(t-\tau} &\text{ for } \tau< t \leq T
\end{align*}
The total queue length corresponding to the three parts are:
\begin{align*}
    \int_0^{t_{2,q_0}^*} q_1(t) dt &= \xi_2 t_{2,q_0}^* + \frac{q_0-\xi_2}{\mu}(1-e^{-\mu t_{2,q_0}^*}) \nonumber\\
    \int_{t_{2,q_0}^*}^{\tau} q_2(t) dt &= \xi_1(\tau- t_{2,q_0}^*) + \frac{c-\xi_2}{\theta}(1-e^{-\theta(\tau - t_{2,q_0}^*)}\\
    \int_{\tau}^T q_3(t) dt &= \xi_2 \cdot (T-t_{2,q_0}^*) + \frac{bq_1(\tau)-\xi_2}{\mu} (1-e^{-\mu (T-\tau)}) \nonumber 
\end{align*}

Hence, taking derivative gives:
\begin{align*}
    \frac{d}{d\tau} \int_0^\tau q_1(t) dt &= 0\\
    \frac{d}{d\tau} \int_{t_{2,q_0}^*}^{\tau} q_2(t) dt &= \xi_1 + (c-\xi_2)e^{-\theta(\tau-t_{2,q_0}^*)}\\
    \frac{d}{d\tau}\int_{\tau}^T q_3(t) dt  &= -\xi_2 - \frac{\theta}{\mu}b(c-\xi_1)e^{-\theta(\tau-t_{2,q_0}^*)}-(b\xi_1-\xi_2)e^{-\mu(T-\tau)} - b(c-\xi_1)\frac{\mu-\theta}{\mu}e^{(\mu-\theta)\tau - \mu T + \theta t_{2,q_0}^*}
\end{align*}
Similar to the argument in Section \ref{apdx: poly}, this derivative has a solution and is only numerically solvable for general parameters.

\subsubsection{$\tau \in I_{2,5} =  \left[0,T\right]$ with $T \leq t_{2,q_0}^*$}

This section corresponds to the case where the system never reaches its capacity $c$. In this case,
\begin{align*}
    q_1(t) &= \xi_2 + (q_0 - \xi_2) e^{-\mu t}  &\text{ for } 0 \leq t \leq \tau \\
    q_2(t) &= \xi_2 + \bigg( bq_1(\tau) - \xi_2 \bigg) e^{-\mu(t-\tau)} \\
    &= \xi_2 + (b-1)\xi_2 e^{-\mu(t-\tau)} + b(q_0-\xi_2)e^{-\mu t}. &\text{ for } \tau < t \leq T
\end{align*}
Taking integral yields the total queue length:
\begin{align}
    \int_0^\tau q_1(t) dt &= \xi_2 \tau + \frac{q_0-\xi_2}{\mu}(1-e^{-\mu\tau}) \nonumber\\
    \int_\tau^T q_2(t) dt &= \xi_2(T-\tau) + \frac{(b-1)\xi_2}{\mu} (1- e^{-\mu(T-\tau)}) + \frac{b(q_0-\xi_1)}{\mu} (e^{-\mu \tau}-e^{-\mu T}).
\end{align}
Taking the derivatives gives:
\begin{align*}
    \frac{d}{d\tau} \int_0^\tau q_1(t) dt &= \xi_2 + (q_0 - \xi_2) e^{-\mu \tau} \\
    \frac{d}{d\tau} \int_\tau^T q_2(t) dt &= -\xi_2 - (b-1)\xi_2 e^{-\mu(T-\tau)} - b(q_0-\xi_1)e^{-\mu \tau}\\
    \frac{d}{d\tau} \int_0^T q(t) dt &= \frac{d}{d\tau} T \cdot J(\tau) \\
    &=\frac{d}{d\tau} \int_0^\tau q_1(t) dt  + \frac{d}{d\tau} \int_\tau^T q_2(t) dt\\
    &= (1-b)(q_0-\xi_1)e^{-\mu \tau} - (b-1) \xi_2e^{-\mu(T-\tau)}
\end{align*}
Observing that $ (1-b)(q_0-\xi_1) < 0$, $b-1<0$, $\frac{d}{d\tau} \int_0^T q(t) dt = 0 $ has solution:
\begin{align*}
    \ln(\xi_2-q_0)-\mu \tau = \ln(\xi_2) -\mu T + \mu \tau\\
    \tau_L^* = \frac{T}{2} + \frac{\ln(1-q_0/\xi_2)}{2\mu}
\end{align*}
with simple convexity analysis verifying that $\tau_L^*$ is indeed the minimizer.

\subsection{\text{$q_0 > c, \lambda \leq \mu c$}}

We first analyze the interval division as in the previous section. Under this section's assumptions, we have $q_0 > c \geq \xi_1$, so we can obtain:
\begin{align*}
    \xi_1 + (q_0-\xi_1)e^{-\theta \tau} > \frac{c}{b}\\
    e^{-\theta \tau}  > \frac{c/b-\xi_1}{q_0-\xi_1}\\
    -\theta\tau > \ln\left( \frac{c/b-\xi_1}{q_0-\xi_1}\right)\\
    \tau < \frac{\ln\left( \frac{q_0-\xi_1}{c/b-\xi_1}\right)}{\theta}
\end{align*}
Here, the interval $I_{3,1} = \left[0,\frac{\ln\left( \frac{q_0-\xi_1}{c/b-\xi_1}\right)}{\theta}\right]$ corresponds to the case where the system is not only above capacity, i.e. $q_1(t) > c$, but also above capacity even after impulse is applied, i.e. $b \cdot q_1(\tau) > c$. If the time horizon is not long enough for the system after jump to recover to capacity ($T \leq \frac{\ln\left( \frac{q_0-\xi_1}{c/b-\xi_1}\right)}{\theta}$), there will only be one interval $I_{3,4} = [0,T]$. On the other hand, if $T > \frac{\ln\left( \frac{q_0-\xi_1}{c/b-\xi_1}\right)}{\theta}$, the rest of the interval will be further divided into two: $I_{3,2} = \left[\frac{\ln\left( \frac{q_0-\xi_1}{c/b-\xi_1}\right)}{\theta}, t_{1,q_0}^*\right], I_{3,3} = \left[t_{1,q_0}^*, T\right]$. Note that discussions of $T \in \left[\frac{\ln\left( \frac{q_0-\xi_1}{c/b-\xi_1}\right)}{\theta}, t_{1,q_0}^*\right]$ and $T \in \left[t_{1,q_0}^*, T\right]$ are merged into one because of the structure of the underlying structure of Erlang-A that allows us to curtail the amount of computation needed. More specifically, because
\begin{align*}
 \tau_L^* = 
    \begin{cases}
       \underset{\tau}{\mathrm{argmin}}_{\tau \in \{\tau_{I_{3,1}}^*,\tau_{I_{3,2}}^*\}} \frac{1}{T} J(\tau) & \text{for }T \in \left[\frac{\ln\left( \frac{q_0-\xi_1}{c/b-\xi_1}\right)}{\theta}, t_{1,q_0}^*\right] \\
        \underset{\tau}{\mathrm{argmin}}_{\tau \in \{\tau_{I_{3,1}}^*,\tau_{I_{3,2}}^*,\tau_{I_{3,3}}^*\}} \frac{1}{T} J(\tau) & \text{for }T \in \left[t_{1,q_0}^*, T\right].
    \end{cases},
\end{align*}
we are able to merge the discussions of $T \in \left[\frac{\ln\left( \frac{q_0-\xi_1}{c/b-\xi_1}\right)}{\theta}, t_{1,q_0}^*\right]$ and $T \in \left[t_{1,q_0}^*, T\right]$ into one.

\subsubsection{$\tau \in I_{3,1} = \left[0,\frac{\ln\left( \frac{q_0-\xi_1}{c/b-\xi_1}\right)}{\theta}\right]$ with $T \geq \frac{\ln\left( \frac{q_0-\xi_1}{c/b-\xi_1}\right)}{\theta}$}
The solution to Equation \ref{abeqn1} in this case is:
\begin{align*}
    q_1(t)& = \xi_1 + (q_0 - \xi_1) e^{-\theta t}&\text{ for } 0 \leq t < \tau\\
    q_2(t)& = \xi_1 + (bq_1(\tau) - \xi_1) e^{-\theta (t-\tau)} &\text{ for } \tau \leq t < \tau + t_{1,bq_1(\tau)}^*\\
    q_3(t)& =  \xi_2 + \left( c- \xi_2 \right) e^{-\mu (t - \tau-t_{1,bq_1(\tau)}^*)} &\text{ for } \tau + t_{1,bq_1(\tau)}^* \leq t \leq T
\end{align*}

Taking integral to compute the four components of $T\cdot J(\tau)$:
\begin{align*}
    \int_0^\tau q_1(t) dt &= \xi_1\tau + \frac{q_0 - \xi_1}{\theta}(1 - e^{-\theta t})\\
    \int_\tau^{\tau + t_{1,bq_1(\tau)}^*} q_2(t) dt &= \xi_1  t_{1,bq_1(\tau)}^* + \frac{bq_1(\tau) - \xi_1}{\theta} \left( 1 - e^{-\theta  t_{1,bq_1(\tau)}^*} \right)\\
    &=\xi_1  t_{1,bq_1(\tau)}^* + \frac{bq_1(\tau) - \xi_1}{\theta} - \frac{c-\xi_1}{\theta} \\
    \int_{\tau + t_{1,bq_1(\tau)}^*}^T q_3(t) dt &= \xi_2(T - \tau-t_{1,bq_1(\tau)}^*) + \frac{c-\xi_2}{\mu} (1- e^{-\mu(T - \tau-t_{1,bq_1(\tau)}^*)}).
\end{align*}

Taking derivatives gives:
\begin{align*}
\frac{d }{d \tau}\int_0^\tau q_1(t) dt &=  \xi_1 + (q_0-\xi_1)e^{-\theta \tau}\\
\frac{d }{d \tau}\int_\tau^{\tau + t_{1,bq_1(\tau)}^*} q_2(t) dt &= \xi_1 (t_{1,bq_1(\tau)}^*)' - b(q_0-\xi_1)e^{-\theta t}\\
\frac{d }{d \tau} \int_{\tau + t_{1,bq_1(\tau)}^*}^T q_3(t) dt &= -\xi_2 (1+ (t_{1,bq_1(\tau)}^*)') + (c-\xi_2) (1+ (t_{1,bq_1(\tau)}^*)')e^{-\mu(T - \tau-t_{1,bq_1(\tau)}^*)}\\
\frac{d }{d \tau} T\cdot J(\tau) &= (1+ (t_{1,bq_1(\tau)}^*)')(\xi_1-\xi_2) + (1-b) (q_0-\xi_1)e^{-\theta \tau} + (c-\xi_2)(1+ (t_{1,bq_1(\tau)}^*)') e^{-\mu(T-\tau - t_{1,bq_1(\tau)}^*)}
\end{align*}
where
\begin{align*}
    (t_{1,bq_1(\tau)}^*)' &= \frac{b\theta (\xi_1-q_0)e^{-\theta \tau}}{\theta (b \xi_1 + b(q_0-\xi_1)e^{-\theta \tau})-\lambda+\mu c -\theta c}\\
    &= \frac{b (\xi_1-q_0)e^{-\theta \tau}}{bq_1(\tau)-\xi_1}
\end{align*}
from which we can observe that
$$1+(t_{1,bq_1(\tau)}^*)' = \frac{(b-1)\xi_1}{bq_1(\tau)-\xi_1} < 0$$

Notice that $(1-b)(c-\xi_1) < 0$ and $(q-b)\xi_1 > 0$, so $\frac{d}{d\tau} T \cdot J(\tau) = 0$ has a solution:
\begin{align*}
    -(c-\xi_1) e^{-\theta(\tau-t_{2,bq_2(\tau)}^{*})} =\xi_1 e^{-\theta(T-\tau)}\\
    \ln(\xi_1-c) -\theta\tau + \theta t_{2,bq_2(\tau)}^{*} = \ln(\xi_1) -\theta T+\theta\tau\\
    \tau^{OPT} = \frac{T+t_{2,bq_2(\tau)}^{*}}{2} + \frac{\ln\left(\frac{\xi_1-c}{\xi_1}\right)}{2\theta}
\end{align*}
From the above one can notice that $(1-b) (q_0-\xi_1) > 0$ and $(c-\xi_2)(1+ (t_{1,bq_1(\tau)}^*)') <0$. So the function $\frac{d }{d \tau} T\cdot J(\tau)$ decreases as $\tau$ increases. In addition, it is easy to verify that $\frac{d }{d \tau} T\cdot J \left( \frac{\ln\left( \frac{q_0-\xi_1}{c/b-\xi_1}\right)}{\theta}\right)$ is positive, it follows that $\frac{d }{d \tau} T\cdot J(\tau)$ is positive for all $\tau \in \left[0,\frac{\ln\left( \frac{q_0-\xi_1}{c/b-\xi_1}\right)}{\theta}\right]$ and hence $\frac{1}{T}J(\tau)$ increases monotonically on the interval $\left[0,\frac{\ln\left( \frac{q_0-\xi_1}{c/b-\xi_1}\right)}{\theta}\right]$, and
\begin{align*}
    \tau_L^* = 0, \quad \quad \tau_U^* = \frac{\ln\left( \frac{q_0-\xi_1}{c/b-\xi_1}\right)}{\theta}.
\end{align*}

\subsubsection{$\tau \in I_{3,2} = \left[\frac{\ln\left( \frac{q_0-\xi_1}{c/b-\xi_1}\right)}{\theta}, t_{1,q_0}^*\right]$ with $T \geq \frac{\ln\left( \frac{q_0-\xi_1}{c/b-\xi_1}\right)}{\theta}$}
Once again, we write down the solution to Equation \ref{abeqn1}:
\begin{align*}
    q_1(t) &= \xi_1 + (q_0 - \xi_1) e^{-\theta t}  &\text{ for } 0 \leq t \leq \tau \\
    q_2(t) &= \xi_2 + (bq_1(\tau)-\xi_2)e^{- \mu(t-\tau)} &\text{ for } \tau < t \leq T\\
\end{align*}
The total queue length corresponding to the three parts are:
\begin{align}
    \int_0^\tau q_1(t) dt &= \xi_1 \tau + \frac{q_0-\xi_1}{\theta}(1-e^{-\theta\tau}) \nonumber\\
    \int_\tau^T q_2(t) dt &= \xi_2 (T-\tau) +  \frac{b(\xi_1-\xi_2)}{\mu}(1-e^{-\mu (T-\tau)})+\frac{b(q_0-\xi_1)}{\mu} (e^{-\theta \tau }- e^{-\mu T +\mu \tau -\theta \tau}) \nonumber 
\end{align}

Hence, taking derivative gives:
\begin{align*}
    \frac{d}{d\tau} \int_0^\tau q_1(t) dt &= \xi_1 + (q_0 - \xi_1) e^{-\theta \tau} \\
    \frac{d}{d\tau} \int_\tau^T q_2(t) dt &= -\xi_2 - b (\xi_1-\xi_2) e^{-\theta(T-\tau)} - \frac{b\theta(q_0-\xi_1)}{\mu}e^{-\theta \tau}-\frac{b(q_0-\xi_1)(\mu-\theta)}{\theta}e^{-\mu T +\mu \tau -\theta \tau} \\
    \frac{d}{d\tau} T \cdot J(\tau) &=(\xi_1-\xi_2)  - b (\xi_1-\xi_2) e^{-\theta(T-\tau)}-\frac{b(q_0-\xi_1)(\mu-\theta)}{\theta}e^{-\mu T +\mu \tau -\theta \tau} + (1-\frac{b\theta}{\mu})(q_0-\xi_1)e^{-\theta \tau}
\end{align*}
which again, lacks a general closed-form solution and one needs to appeal to numerical ones.

\subsubsection{$\tau \in I_{3,3} = [t_{1,q_0}^*,T]$ with $T \geq \frac{\ln\left( \frac{q_0-\xi_1}{c/b-\xi_1}\right)}{\theta}$}
The solution to Equation \ref{abeqn1} in this case is:
\begin{align*}
    q_1(t)& = \xi_1 + (q_0 - \xi_1) e^{-\theta t}&\text{ for } 0 \leq t < t_{1,q_0}^*\\
    q_2(t)& = \xi_2 + (c - \xi_2) e^{-\mu (t-t_{1,q_0}^*)} &\text{ for } t_{1,q_0}^* \leq t < \tau\\
    q_3(t)& =  \xi_2 + \left( bq_2(\tau)- \xi_2 \right) e^{-\mu (t - \tau)} &\text{ for } \tau\leq t \leq T
\end{align*}

Taking integral to compute the four components of $T\cdot J(\tau)$:
\begin{align*}
    \int_0^{t_{1,q_0}^*} q_1(t) dt &= \xi_1t_{1,q_0}^* + \frac{q_0 - \xi_1}{\theta}(1 - e^{-\theta t})\\
    \int_{t_{1,q_0}^*}^{\tau } q_2(t) dt &= \xi_2 (\tau - t_{1,q_0}^*) + \frac{c - \xi_2}{\mu} \left( 1 - e^{-\mu (\tau-t_{1,q_0}^*)} \right)\\
    \int_{\tau }^T q_3(t) dt &= \xi_2(T - \tau) + \frac{(b-1)\xi_2}{\mu} (1- e^{-\mu(T - \tau)}) + \frac{b(c-\xi_2)}{\mu}(e^{-\mu(\tau-t_{1,q_0}^*)}-e^{-\mu(T-t_{1,q_0}^*)}).
\end{align*}

Taking derivatives gives:
\begin{align*}
\frac{d }{d \tau}\int_0^{t_{1,q_0}^*} q_1(t) dt &=  0\\
\frac{d }{d \tau}\int_{t_{1,q_0}^*}^{\tau } q_2(t) dt &= \xi_2+(c-\xi_2)e^{-\mu (\tau-t_{1,q_0}^*)}\\
\frac{d }{d \tau} \int_{\tau }^T q_3(t) dt &= -\xi_2 -(b-1)\xi_2e^{-\mu(T - \tau)} - b(c-\xi_2)e^{-\mu(\tau-t_{1,q_0}^*)}\\
\frac{d }{d \tau} T\cdot J(\tau) &= (1-b)(c-\xi_2)e^{-\mu(\tau-t_{1,q_0}^*)} +(1-b)\xi_2e^{-\mu(T - \tau)}
\end{align*}
By the assumption of this section, we can conclude that since $(1-b)(c-\xi_2)>0$, $(1-b)\xi_2>0$, $\frac{d }{d \tau} T\cdot J(\tau) >0$ for all $\tau \in [t_{1,q_0}^*,T]$ and thus $\frac{1}{T}T(\tau)$ monotonically increase over the interval $[t_{1,q_0}^*,T]$. Hence, the minimizer on this interval is:
$$\tau_L^* = t_{1,q_0}^*.$$
The argument for $\tau_U^*=T$ follows from the previous sections.

\subsubsection{$\tau \in I_{3,4} = [0,T]$ with $T \leq \frac{\ln\left( \frac{q_0-\xi_1}{c/b-\xi_1}\right)}{\theta}$}

In this case, we again first write down the solution to Equation \ref{abeqn1}:
\begin{align*}
    q_1(t) &= \xi_1 + (q_0 - \xi_1) e^{-\theta t}  &\text{ for } 0 \leq t \leq \tau \\
    q_2(t) &= \xi_1 + (b-1) \xi_1 e^{-\theta(t-\tau)}+b(q_0-\xi_1)e^{-\theta t} &\text{ for } \tau < t \leq T\\
\end{align*}
The total queue length corresponding to the three parts are:
\begin{align}
    \int_0^\tau q_1(t) dt &= \xi_1 \tau + \frac{q_0-\xi_1}{\theta}(1-e^{-\theta\tau}) \nonumber\\
    \int_\tau^T q_2(t) dt &= \xi_1 (T-\tau) +  \frac{bq_1(\tau)-\xi_1}{\theta}(1-e^{-\theta (T-\tau)}) \nonumber 
\end{align}

Hence, taking derivative gives:
\begin{align*}
    \frac{d}{d\tau} \int_0^\tau q_1(t) dt &= \xi_1 + (q_0 - \xi_1) e^{-\theta \tau} \\
    \frac{d}{d\tau} \int_\tau^T q_2(t) dt &= -\xi_1 - (b-1) \xi_1 e^{-\theta(T-\tau)} - b(q_0-\xi_1)e^{-\theta \tau}\\
    \frac{d}{d\tau} T \cdot J(\tau) &= (1-b)(q_0-\xi_1)e^{-\theta \tau}-(b-1)\xi_1 e^{-\theta(T-\tau)}
\end{align*}
Observe that $ (1-b)(q_0-\xi_1)>0$ and $-(b-1)\xi_1>0$. Thus $ J(\tau)$ is monotonically increasing, thus, $\tau_L^* = 0,\tau_U^* = T$.

\subsection{$q_0 \leq c, \lambda \leq \mu c$}

Since there is only one sub-case, we write down the solution for the differential equation \ref{abeqn1} for $\tau \in \{0,T\}$:

\begin{align*}
    q_1(t) &= \xi_2 + (q_0 - \xi_2) e^{-\mu t}  &\text{ for } 0 \leq t \leq \tau \\
    q_2(t) &= \xi_2 + \bigg( b\xi_2 + b(q_0-\xi_2)e^{-\mu \tau} - \xi_2 \bigg) e^{-\mu(t-\tau)} \\
    &= \xi_2 + (b-1) \xi_2 \cdot e^{-\mu(t-\tau)} + b(q_0-\xi_2)e^{-\mu t}  &\text{ for } \tau < t \leq T
\end{align*}

Similar to the previous section we discuss two cases:

(i) When $q_0 > \xi_2$. Under this case, we
can derive the derivative of the total queue length $\frac{d}{d\tau} T\cdot J(\tau)$ to be
$$\frac{d}{d\tau} \int_0^{\tau} q_1(t) dt +\frac{d}{d\tau} \int_{\tau}^T q_2(t, \tau) dt = (1-b)(q_0-\xi_2)e^{-\mu \tau} - (b-1) \xi_2  e^{-\mu (T-\tau)},$$
is strictly positive, making 
$$\tau_L^* = \underset{\tau}{\mathrm{argmin}}_{\tau} \int_0^t q(t) dt = 0,$$
$$\tau_U^* =\underset{\tau}{\mathrm{argmax}}_{\tau} \int_0^t q(t) dt = T.$$

(ii) When $q_0 \leq \xi_2$, set derivative to 0:
\begin{align*}
    \tau^{OPT} = \frac{\ln(\xi_2-q_0)-\ln(\xi_2)}{2 \mu}+ \frac{t}{2},
\end{align*}
Simple convexity check confirms that:
$$\tau_L^* = \underset{\tau}{\mathrm{argmin}} \int_0^t q(t) dt = \tau^{OPT},$$
$$\tau_U^* =  \underset{\tau}{\mathrm{argmax}} \int_0^t q(t) dt = T.$$

\bibliographystyle{plainnat}
\bibliography{refs}

\end{document}